%% file: artikel1.tex
\documentclass[a4paper,12pt]{article}

\usepackage{amsmath}
\usepackage{mathtools}
\usepackage{amsfonts}
\usepackage{amssymb}
\usepackage{amsthm}
\usepackage{graphicx}
\usepackage{psfrag}
\usepackage{subfigure}

\usepackage[english]{babel}
\usepackage[latin9]{inputenc}
\usepackage[T1]{fontenc}
\usepackage{amssymb}
\usepackage{xcolor}
\usepackage[varumlaut]{yfonts}
\usepackage{amsmath}
\usepackage{mathtools}
\usepackage{amsopn}
\usepackage{amsthm}
\usepackage{ulem}
\usepackage{setspace}
\usepackage{esint}
\usepackage{graphicx}
\usepackage{babel}
\usepackage{relsize}
\usepackage{float}
\usepackage[labelfont={bf}]{caption}
\usepackage{nicefrac}
\usepackage{paralist}
\usepackage{verbatim}
\usepackage{enumitem}
\usepackage{soul}

\newcommand{\dd}{\ensuremath{\, \mathrm{d}}}

\newtheorem{definition}{Definition}[section]
\newtheorem{rem}[definition]{Remark}

\newtheorem{lem}[definition]{Lemma}

\newtheorem{satz}[definition]{Theorem}

\newtheorem{claim}[definition]{Claim}
\newtheorem{cor}[definition]{Corollary}
\newtheorem{prop}[definition]{Proposition}

\numberwithin{equation}{section}  

\begin{document}
\begin{titlepage}
\title{Short Time Existence for the Curve Diffusion Flow with a Contact Angle}
\author{  Helmut Abels\footnote{Fakult\"at f\"ur Mathematik,  
Universit\"at Regensburg,
93040 Regensburg,
Germany, e-mail: {\sf helmut.abels@mathematik.uni-regensburg.de}}\ \ and Julia Butz\footnote{Fakult\"at f\"ur Mathematik,  
Universit\"at Regensburg,
93040 Regensburg,
Germany, e-mail: {\sf julia4.butz@mathematik.uni-regensburg.de}}}
\end{titlepage}
\maketitle
\begin{abstract}
We show short-time existence for curves driven by curve diffusion flow with a prescribed contact angle $\alpha \in (0, \pi)$: The evolving curve has free boundary points, which are supported on a line and it satisfies a no-flux condition. The initial data are suitable curves of class $W_2^{\gamma}$ with $\gamma \in (\tfrac{3}{2}, 2]$. For the proof the evolving curve is represented by a height function over a reference curve: The local well-posedness of the resulting quasilinear, parabolic, fourth-order PDE for the height function is proven with the help of contraction mapping principle. Difficulties arise due to the low regularity of the initial curve. To this end, we have to establish suitable product estimates in time weighted anisotropic $L_2$-Sobolev spaces of low regularity for proving that the non-linearities are well-defined and contractive for small times.
\end{abstract}

\noindent{\bf AMS-Classification:} 53C44, 35K35, 35K55

\section{Introduction and Main Result} \label{intro}

The curve diffusion flow is the one-dimensional version of the surface diffusion flow, which describes the motion of interfaces in the case that it is governed purely by diffusion within the interface. The geometric evolution equation was originally derived by Mullins to model the development of surface grooves at the grain boundaries of a heated polycrystal in 1957, see \cite{mullins}. It turns out, that curve diffusion flow is the $H^{-1}$-gradient flow of the length of the curves, see \cite{garcke}. Moreover, the flow is related to the Cahn-Hilliard equation for a degenerate mobility: This equation arises in material science and models the phase separation of a binary alloy, which separates and forms domains mainly filled by a single component. Formal asymptotic expansions suggest that surface diffusion flow is the singular limit of the Cahn-Hilliard equation with a degenerate mobility for the case that the interfacial layer does not intersect the boundary of the domain, see \cite{cahnelliottnovick}. Garcke and Novick-Cohen considered also the situation of an intersection of the interfacial layer with the external boundary and identified formally the sharp interface model, where the interfaces evolve in the two-dimensional case according to the curve diffusion flow and subject to an attachment condition, a $\tfrac{\pi}{2}$-angle condition, and a no flux-condition, see \cite{garckenovick}. They also established a short time existence result for this problem in the case of initial data of class $C^{4 + \alpha}$, see \cite{garckenovick}. This is related to the subject of this work: In the following, we want to prove a short-time existence result for the curve diffusion flow of open curves in the case of rough initial data.

More precisely, we look for a time dependent family of regular curves $\Gamma := \{\Gamma_t\}_{t \geq 0}$ satisfying
\begin{align}
V = - \partial_{ss}\kappa_{\Gamma_t} \label{g1} && \textrm{ on } \Gamma_t, t > 0,
\end{align}
where $V$ is the scalar normal velocity, $\kappa_{\Gamma_t}$ is the scalar curvature of $\Gamma_t$, and $s$ denotes the arc length parameter. We complement the evolution law with the boundary conditions  
\begin{align}
\partial \Gamma_t &\subset \mathbb{R} \times \{0\} && \textrm{ for } t > 0, \label{g11} \\
\measuredangle \left({n}_{\Gamma_t},  \begin{pmatrix} 0 \\ -1\end{pmatrix} \right) &= \pi - \alpha && \textrm{ at } \partial \Gamma_t \textrm{ for } t > 0, \label{g2}\\
\partial_s \kappa_{\Gamma_t} &= 0  && \textrm{ at } \partial \Gamma_t \textrm{ for } t > 0, \label{g3}
\end{align}
where ${n}_{\Gamma_t}$ is the unit normal vector of $\Gamma_t$ and $\alpha \in (0, \pi)$. A sketch is given in Figure \ref{evol}. 
\begin{center}\vspace{0 cm}
	\scalebox{1}{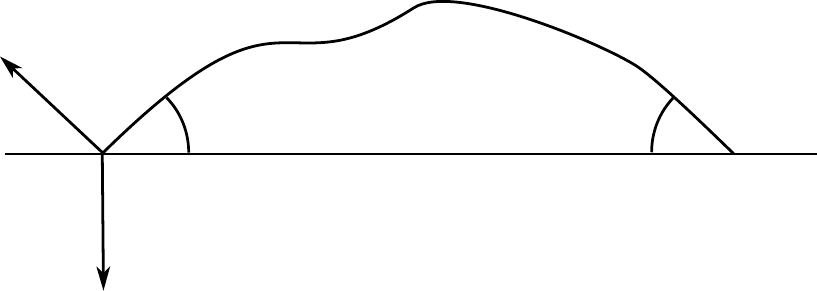}
	\captionof{figure}{Evolution by curve diffusion flow with $\alpha$-angle condition for $\alpha < \tfrac{\pi}{2}$.}
	\label{evol}
\end{center}\vspace{0 cm}

Even though Garcke, Ito, and Kohsaka proved a global existence result for initial data sufficiently close to an equilibrium for a $\tfrac{\pi}{2}$-angle condition in \cite{garitokoh}, Escher, Mayer, and Simonett gave numerical evidence that closed curves in the plane, which are moving according to \eqref{g1} can develop singularities in finite time, cf.\ \cite{eschmaysim}. Indeed, for smooth closed curves driven by \eqref{g1}, Chou provided a sharp criterion for a finite lifespan of the flow in \cite{chou}. Additionally, Chou, see \cite{chou}, and Dzuik, Kuwert, and Schätzle, see \cite{dziuk}, showed that if a solution has a maximal lifespan $T_{max} < \infty$, then the $L_2$-norm of the curvature with respect to the arc length parameter tends to infinity as $T_{max}$ is approached. Moreover, they gave a rate for the blow-up. A blow-up criterion and criteria for global existence for the curve diffusion flow with a free boundary were proven by Wheeler and Wheeler in \cite{wheelerwheeler}. Related results for different flows of open curves were shown by Lin in \cite{lin}, Dall'Acqua and Pozzi in \cite{acquapozzi}, and Dall'Acqua, Lin, and Pozzi in \cite{acqualinpozzi}: The authors consider $L_2$-gradient flows of the bending energy, either with length penalization or with length constraint, for open curves and deduce global existence and subconvergence results.\\

In order to obtain a short time existence result, the following strategy is pursued: For a fixed reference curve and coordinates, we represent the evolving curve, which is in some sense close to the reference curve, by a height function. Thus, the geometric evolution equation \eqref{g1}-\eqref{g3} is reduced to a quasilinear parabolic partial differential equation for the height function on a fixed interval, at least as long as the curve is sufficiently close to the initial curve. The standard approach to attack these kind of problems is a contraction mapping argument: First, the equation is linearized and the function spaces for the solution and the data of the linearized system are chosen such that the linear problem can be solved with optimal regularity and the nonlinear terms are contractive for small times. The key ingredient for solving the linear problem in our situation is a result by Meyries and Schnaubelt on maximal $L_p$-regularity with temporal weights, \cite{meyries}: We deduce that for admissible initial data in $W^{4(\mu - \nicefrac{1}{2})}_{2} (I)$, $\mu \in (\tfrac{7}{8}, 1]$, there exists a unique solution of the linear problem in the temporal weighted parabolic space $W^1_{2, \mu} ((0, T); L_{2} (I) \cap L_{2, \mu} ((0, T); W^{4}_{2} (I))$ for any $T>0$, see Section \ref{pre} for the definitions of the spaces. In order to show that the nonlinear terms are contractive for small times, it is crucial to study the structure of the non-linearities and to establish suitable product estimates in time weighted anisotropic $L_2$-Sobolev spaces of low regularity. This enables us to apply Banach's fixed point theorem and we obtain a unique solution to the partial differential equation. \\

In the following, $\Phi^*([0, 1])$, $\Phi^*:[0, 1] \rightarrow \mathbb{R}^2$, which fulfills suitable boundary conditions, serves as a reference curve and $\rho: [0, T] \times [0, 1] \rightarrow (-d, d)$ is a height function such that $\Gamma$ can be expressed by a curvilinear coordinate system, see Section \ref{rog} for details. The main result reads:
\begin{satz}[Local Well-Posedness for Data Close to a Reference Curve]~\\ \label{local}
  Let $\Phi^*$ and $\eta$ be given such that \eqref{bound1} and \eqref{eta} below are fulfilled, respectively, and let $T_0>0$ and $R_2>0$ such that $\|\mathcal{L}^{-1} \|_{L(\mathbb{E}_{0, \mu} \times \tilde{\mathbb{F}}_{\mu} \times X_{\mu};  \mathbb{E}_{\mu, T_0})}\leq R_2$, cf.\ \eqref{equ} for the Definition of $\mathcal{L}^{-1}$.
	Furthermore, let $\rho_0 \in W_2^{4\left(\mu - \nicefrac{1}{2}\right)} (I)$, with $I = (0, 1)$ and $\mu \in \left(\frac78, 1 \right]$, fulfill the conditions
	\begin{align}
	\|\rho_0\|_{C(\bar{I})} < \frac{K_0}{3}\qquad   \textrm{ and } \qquad  \|\partial_\sigma \rho_0\|_{C(\bar{I})} < \frac{K_1}{3},
	\label{small}
	\end{align}
	and the compatibility condition
	\begin{align}
	\partial_{\sigma} \rho_0 (\sigma) = 0 && \textrm{ for } \sigma \in \{0, 1\},
	\label{komp}
	\end{align}
	where $K_0$, $K_1$ are specified in Lemma \ref{regu}. Moreover, let $R_1>0$ be such that  $\|\rho_0\|_{X_{\mu}} \leq R_1$. 
	Then there exists a $T = T(\alpha, \Phi^*, \eta, R_1, R_2) \in (0,T_0]$   such that the problem 
	\begin{align}
	\rho_t &= - \frac{J(\rho)}{\langle \Psi_q, R \Psi_\sigma \rangle} \Delta (\rho) \kappa(\rho) && \textrm{ for } \sigma \in (0, 1) \textrm{ and } t \in J, \nonumber \\
	\partial_\sigma \rho (t, \sigma) &= 0  && \textrm{ for } \sigma \in \{0, 1\} \textrm{ and } t \in J, \nonumber \\
	\partial_\sigma \kappa(\rho) &= 0  && \textrm{ for } \sigma \in \{0, 1\} \textrm{ and } t \in J \label{prob}
	\end{align}
	possesses a unique solution $\rho \in W^1_{2, \mu} (J; L_{2} (I)) \cap L_{2, \mu} (J; W^{4}_{2} (I))$, $J = (0, T)$ such that $\rho(t)$ satisfies the bounds \eqref{kleinrho} and \eqref{kleindrho} below for all $t \in [0, T]$, and $\rho (\cdot, 0) = \rho_0$ in $ W_2^{4\left(\mu - \nicefrac{1}{2}\right)} (I)$.
\end{satz}

We give a sketch of the situation in Figure \ref{sketch}. 
\begin{center} \vspace{0,5 cm}
	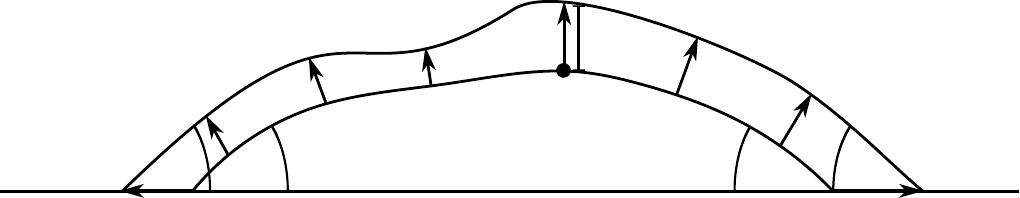
	\captionof{figure}{Representation of a curve by a reference curve $\Phi*$, curvilinear coordinates, and a height function $\rho(t, \sigma)$.}
	\label{sketch}
\end{center}

We use the time weighted approach, as it enables us to apply the result to initial data with flexible regularity. Moreover, we can directly exploit the smoothing properties of parabolic equations, see Chapter 2 in \protect{\cite{meyries}}.

Note that in this work, we establish short time existence for curves which can be described as a graph over the reference curve with a height function, which is small in some sense. The corresponding result which allows for starting the flow for a fixed initial curve can be found in Theorem 4.1.3 in \protect{\cite{butzdiss}} and \cite{butzprint2}, where it is also used to establish a blow-up criterion for the geometric evolution equation \eqref{g1} -\eqref{g3}, cf.\ Theorem 4.1.4 in \protect{\cite{butzdiss}}. The latter result requires that the flow starts for initial data which are less regular than $W^2_2(I)$. \\

This article is organized as follows: In Section \ref{pre}, we summarize some notation and preliminary results. In the following, we present the main steps of the local well-posedness result for rough initial data in $W^{4(\mu - \nicefrac{1}{2})}_{2} (I)$ for $\mu \in (\tfrac{7}{8}, 1]$, cf.\ Theorem \ref{local}: In Section \ref{rog} we derive the partial differential equation (PDE) which corresponds to the geometric evolution equation (GEE) \eqref{g1} -\eqref{g3}. Then, in Section \ref{line} the linear problem is treated. Afterwards, we prove the main result in Section \ref{cont}.

\section{Preliminaries and Fundamental Mathematical Tools} \label{pre} 

In this section, we want to present some preliminary results on fractional Sobolev spaces and maximal $L_2$-regularity. Additionally, a geometrical estimate is given.\\ 

\subsection{Fractional Sobolev Spaces and Some Properties} 

The following spaces will be crucial for our setting. Most facts and definitions stated in the first part of this section are derived in \protect{\cite{meyries_inter}}. For more results about these spaces, e.g.\ dense subsets, extension operators, trace theorems and embeddings, the reader is referred to \cite{meydiss} and \protect{\cite{meyries_inter}}. \\
General properties of real and complex interpolation theory can be found in \protect{\cite{lunardi}} or \protect{\cite{triebel}}.

\begin{definition} [Weighted Lebesgue Space]
	\label{defleb}
	Let $J = (0, T)$, $0 < T \leq \infty$ and $E$ be a Banach space. For $1 < p < \infty$ and $\mu \in \left( \frac{1}{p}, 1 \right]$ the \textbf{weighted Lebesgue space} is given by
	\begin{align*}
	L_{p, \mu}(J; E) := \left\{u: J \to E \textrm{ is strongly measurable}: \|u\|_{L_{p, \mu}(J; E)} <  \infty \right\},
	\end{align*} 
	where
	\begin{align*}
	\|u\|_{L_{p, \mu}(J; E)} := \left\| \left[t \mapsto t^{1- \mu} u(t) \right] \right\|_{L_{p}(J; E)} = \left( \int_{J} t^{(1-\mu)p} \|u(t)\|_E^{p} \dd t \right)^{\frac{1}{p}}.
	\end{align*}
\end{definition}

\begin{rem} \label{nummer}
	\begin{enumerate}
		\item $(L_{p, \mu}(J; E), \|u\|_{L_{p, \mu}(J; E)})$ is a Banach space.
		\item One easily sees that for $T < \infty$ it follows
		\begin{align*}
		L_{p}(J; E) \hookrightarrow L_{p, \mu}(J; E).
		\end{align*}
		This does not hold true for $T = \infty$.
		\item We have $L_{p, \mu}((0, T); E) \subset L_{p}((\tau, T); E)$ for $\tau \in (0, T)$.
		\item For $\mu = 1$ it holds $L_{p,1}(J; E) = L_{p}(J; E)$.
	\end{enumerate}
\end{rem}

Moreover, we define associated weighted Sobolev spaces.
\begin{definition} [Weighted Sobolev Space]
	\label{defsobo}
	Let $J = (0, T)$, $0 < T \leq \infty$ and $E$ be a Banach space. For $1 \leq p < \infty$, $k \in \mathbb{N}_0$, and $\mu \in \left( \frac{1}{p}, 1 \right]$ the \textbf{weighted Sobolev space} is given by
	\begin{align*}
	W^k_{p, \mu}(J; E) = H^k_{p, \mu}(J; E) := \left\{u \in W^k_{1, \textrm{loc}} (J; E): u^{(j)} \in L_{p, \mu}(J; E) \textrm{ for } \{0, \cdots, k\}  \right\}
	\end{align*}
	for $k \neq 0$, where $u^{(j)} := \big(\frac{\dd}{\dd t} \big)^j u$, and we set $W^0_{p, \mu}(J; E) := L_{p, \mu}(J; E)$. We equip it with the norm
	\begin{align*}
	\|u\|_{W^k_{p, \mu}(J; E)} := \left(\sum_{j=0}^{k} \left\|u^{(j)} \right\|^p_{L_{p, \mu}(J; E)} \right)^{\frac{1}{p}}.
	\end{align*}
\end{definition}

\begin{rem}
	$(W^k_{p, \mu}(J; E), \|u\|_{W^k_{p, \mu}(J; E)})$ is a Banach space, see Theorem in Section 3.2.2 of \protect{\cite{triebel}}.
\end{rem}

In the following, we introduce a generalization of the usual Sobolev spaces by the means of interpolation theory. By $(\cdot,\cdot)_{\theta, p}$ and $(\cdot,\cdot)_{[\theta]}$ we denote real and complex interpolation functor, respectively.
\begin{definition} [Weighted Slobodetskii Space, Weighted Bessel Potential Space]
	\label{defslobo} \label{defbessel}
	Let $J = (0, T)$, $0 < T \leq \infty$ and $E$ be a Banach space. For $1 \leq p < \infty$, $s \in \mathbb{R}^+ \backslash \mathbb{N}$, and $\mu \in \left( \frac{1}{p}, 1 \right]$ the \textbf{weighted Slobodetskii space} and the \textbf{weighted Bessel potential space}, respectively, are given by
	\begin{align*}
	W^s_{p, \mu} (J; E) &:= \left(W^{\lfloor s \rfloor}_{p, \mu}(J; E), W^{\lfloor s \rfloor + 1}_{p, \mu}(J; E)\right)_{s - \lfloor s \rfloor, p}, \\
	H^s_{p, \mu} (J; E) &:= \left(W^{\lfloor s \rfloor}_{p, \mu}(J; E), W^{\lfloor s \rfloor + 1}_{p, \mu}(J; E)\right)_{[s - \lfloor s \rfloor]}.
	\end{align*}
\end{definition}

\begin{rem} \label{nummer1}
	\begin{enumerate} 
		\item Both spaces are Banach spaces by interpolation theory, cf.\ Proposition 1.2.4 in \protect{\cite{lunardi}} and the Theorem in 1.9.1 in \protect{\cite{triebel}}. 
		\item We have $W^s_{p, 1}(J; E) =  W^s_{p}(J; E)$ and $H^s_{p, 1}(J; E) =  H^s_{p}(J; E)$ for all $s \geq 0$.
		\item For $p\in (1, \infty)$ we obtain by Lemma 2.1 in \protect{\cite{meyries_inter}} that the trace $u \mapsto u^{(j)}(0)$ is continuous from $W^k_{p, \mu}(J; E)$ to $E$ for all $j \in \{0, \dots, k-1\}$. Thus, for $k \in \mathbb{N}$ we can define
		\begin{align*}
		{}_0 W^k_{p, \mu}(J; E) = {}_0 H^k_{p, \mu}(J; E) := \left\{u \in W^k_{p, \mu}(J; E) : u^{(j)}(0) = 0 \textrm{ for } j \in \{0, \dots, k-1\} \right\},
		\end{align*}
		which are Banach spaces with the norm of $W^k_{p, \mu}(J; E)$. Moreover, we set for convenience
		\begin{align*}
		{}_0 W^0_{p, \mu}(J; E) = {}_0 H^0_{p, \mu}(J; E) := L_{p, \mu}(J; E).
		\end{align*}
		\item By Proposition 2.10 in \protect{\cite{meyries_inter}}, it follows for $k + 1 - \mu + \nicefrac{1}{p} < s < k + 2 - \mu + \nicefrac{1}{p}$ with $k \in \mathbb{N}_0$
		\begin{align*}
		W^s_{p, \mu} (J; E) &\hookrightarrow BUC^k (\bar{J}; E),\\
		H^s_{p, \mu} (J; E) &\hookrightarrow BUC^k (\bar{J}; E).
		\end{align*}
		If additionally one replaces the spaces $W^s_{p, \mu} (J; E)$ and $H^s_{p, \mu} (J; E)$ by the ${}_0 W^s_{p, \mu} (J; E)$ and ${}_0 H^s_{p, \mu} (J; E)$, respectively, and $s \in [0, 2]$, then the operator norms of the embeddings do not depend on $J$.
		\item We can define the corresponding fractional order spaces ${}_0 W^s_{p, \mu}(J; E)$ and ${}_0 H^s_{p, \mu}(J; E)$ ana\-logously to Definition \ref{defslobo}. By Proposition 2.10 in \protect{\cite{meyries_inter}}, we have for $p \in (1, \infty)$ the characterization
		\begin{align*}
		{}_0 W^s_{p, \mu}(J; E) &= \left\{u \in W^s_{p, \mu}(J; E) : u^{(j)}(0) = 0 \textrm{ for } j \in \{0, \dots, k\} \right\}, \\
		{}_0 H^s_{p, \mu}(J; E) &= \left\{u \in W^s_{p, \mu}(J; E) : u^{(j)}(0) = 0 \textrm{ for } j \in \{0, \dots, k\} \right\},
		\end{align*}
		if $k + 1 - \mu + \nicefrac1p < s < k + 1 + (1 - \mu + \nicefrac1p)$, $k \in \mathbb{N}_0$. 
		\item By equation (2.7) and (2.8) in \protect{\cite{meyries_inter}}, we have for $s = \lfloor s \rfloor + s^*$
		\begin{align*}
		{}_0 W^s_{p, \mu}(J; E) &= \left\{u \in {}_0 W^{\lfloor s \rfloor}_{p, \mu}(J; E) : u^{(\lfloor s \rfloor)} \in {}_0 W^{s^*}_{p, \mu}(J; E) \right\}, \\
		W^s_{p, \mu}(J; E) &= \left\{u \in W^{\lfloor s \rfloor}_{p, \mu}(J; E) : u^{(\lfloor s \rfloor)} \in W^{s^*}_{p, \mu}(J; E) \right\},
		\end{align*}
		where the natural norms are equivalent with constants independent of $J$.
		\item By Proposition 2.10 in \protect{\cite{meyries_inter}}, we obtain that $W^s_{p, \mu}(J; E) = {}_0 W^s_{p, \mu}(J; E)$ for $1 - \mu + \nicefrac{1}{p} > s > 0$.
		\item By interpolation theory, see (2.6) in \protect{\cite{meyries_inter}}, we have the following representation of the Slobodetskii space: For $s \in (0, 1)$ it holds
		\begin{align*}
		W^s_{p, \mu} (J; E) =& \left\{ u \in L_{p, \mu} (J; E): [u]_{W^s_{p, \mu} (J; E)} < \infty \right\},
		\end{align*}
		where 
		\begin{align*}
		[u]_{W^s_{p, \mu} (J; E)} := \left( \int_0^T \int_0^t \tau^{(1 - \mu)p} \frac{\|u(t) - u(\tau)\|_E^p}{|t - \tau|^{1 + sp}} \dd \tau \dd t \right) ^{\frac{1}{p}}.
		\end{align*}
		Then, the norm given by
		\begin{align*}
		\|u\|_{W^s_{p, \mu} (J; E)} := \|u\|_{L_{p, \mu} (J; E)} + [u]_{W^s_{p, \mu} (J; E)} 
		\end{align*}
		is equivalent to the one induced by interpolation.
		\item If $E = \mathbb{R}$ is the image space, we omit it, e.g.\ $W^s_{p, \mu} (J) := W^s_{p, \mu} (J; \mathbb{R})$.
	\end{enumerate}
\end{rem}

\subsection{Embeddings with Uniform Operator Norms} 

The following proposition shows that the operator norms in Theorem \ref{embl1} are uniform in time for all $0 < T \leq T_0 < \infty$ if one uses a suitable norm.
\begin{prop} \label{unab}
	Let $0 < T_0 < \infty$ be fixed and $J = (0, T)$ for $0 < T \leq T_0$. Moreover, let $\mu \in \left(\frac{1}{p}, 1 \right]$ and $E$ be a Banach space. We set for $ s > 1 - \mu + \frac{1}{p}$
	\begin{align}
	\|\rho\|^{'}_{W^{s}_{p, \mu}(J; E)} := \|\rho\|_{W^{s}_{p, \mu}(J; E)} + \|\rho_{| t = 0}\|_{E}. 
	\label{strichnorm}
	\end{align}
	\begin{enumerate}
		\item{\label{unab0} Let $1 < p < q < \infty$, $2 \geq s> \tau \geq 0$, and $s - \frac{1}{p} > \tau - \frac{1}{q}$. \\ 
			Then $W^{s}_{p, \mu}(J; E) \hookrightarrow W^{\tau}_{q, \mu}(J; E)$ with the estimate
			\begin{align*}
			\|\rho\|_{W^{\tau}_{q, \mu}(J; E)} \leq C(T_0) \|\rho\|^{'}_{W^{s}_{p, \mu}(J; E)} && \textrm{ for } s > 1 - \mu + \frac{1}{p}, \\
			\|\rho\|_{W^{\tau}_{q, \mu}(J; E)} \leq C(T_0) \|\rho\|_{W^{s}_{p, \mu}(J; E)} && \textrm{ for } s < 1 - \mu + \frac{1}{p}.
			\end{align*}}
		\item{\label{unab0a} Let $1 < p < q < \infty$, $2 \geq s> \tau \geq 0$, and $s - (1 - \mu) - \frac{1}{p} > \tau - \frac{1}{q}$. \\ 
			Then
			$W^{s}_{p, \mu}(J; E) \hookrightarrow W^{\tau}_{q}(J; E)$ with the estimate
			\begin{align*}
			\|\rho\|_{W^{\tau}_{q}(J; E)} \leq C(T_0) \|\rho\|^{'}_{W^{s}_{p, \mu}(J; E)} && \textrm{ for } s > 1 - \mu + \frac{1}{p}, \\
			\|\rho\|_{W^{\tau}_{q}(J; E)} \leq C(T_0) \|\rho\|_{W^{s}_{p, \mu}(J; E)} && \textrm{ for } s < 1 - \mu + \frac{1}{p}.
			\end{align*}}
		\item{\label{unab0b} Let $1 < p < \infty$, $2 \geq s > 1 - \mu + \frac{1}{p}$, and $\alpha \in (0, 1)$.\\ 
			Then $W^{s}_{p, \mu}(J; E) \hookrightarrow C^{\alpha}(\bar{J}; E)$ for $s - \left(1 - \mu\right) + \frac{1}{p}  > \alpha > 0$ with the estimate
			\begin{align*}
			\|\rho\|_{C^{\alpha}(\bar{J}; E)} \leq C(T_0) \|\rho\|^{'}_{W^{s}_{p, \mu}(J; E)}.
			\end{align*}}
	\end{enumerate}
	Each of the constants $C$ does not depend on $T$.
\end{prop}

\begin{rem} \label{afterunab}
	Based on item \ref{unab0b}, we can also prove the following statement:
	Let $1 < p < \infty$, $k \in \mathbb{N}$. Then $W^{s}_{p}(J; E) \hookrightarrow C^{k, \alpha}(\bar{J}; E)$ for $s - \nicefrac{1}{p} > k + \alpha > 0$ with the estimate
	\begin{align*}
	\|\rho\|_{C^{k, \alpha}(\bar{J}; E)} \leq C \|\rho\|_{W^{s}_{p}(J; E)},
	\end{align*}
	where $C$ depends on $J$. \\
	
	In order to show this, we use the characterization of Slobodetskii spaces in Lemma 1.1.8 in \protect{\cite{meydiss}}. Thus, we can apply the reasoning of the proof of \ref{unab0b} for $\mu = 1$ to $\partial^m_\sigma f \in W^{s-k}_{p}(J; E)$ for $m < s$, since $s - k > 0$.
\end{rem}

The proofs can be found in 2.1.15 in \protect{\cite{butzdiss}}: By using suitable extensions the claims can be reduced to the zero-trace case and the following embedding theorem for fractional Sobolev Spaces can be applied:
\begin{satz} 
	\label{embl1}
	Let $J = (0, T)$ be finite, $1 < p < q < \infty$, $\mu \in \left(\frac{1}{p}, 1\right]$ and $s > \tau \geq 0$ and $E$ be a Banach space. Then
	\begin{align}
	& W^{s}_{p, \mu}(J) \hookrightarrow W^{\tau}_{q, \mu}(J) &&\textrm{ holds if }&& s - \frac{1}{p} > \tau - \frac{1}{q}. \label{emb} \\
	& W^{s}_{p, \mu}(J; E) \hookrightarrow W^{\tau}_{q}(J; E) &&\textrm{ holds if }&& s - (1 - \mu) - \frac{1}{p} > \tau - \frac{1}{q}. \label{emb1}
	\end{align}
	These embeddings remain true if one replaces the $W$-spaces by the $H$-, the ${}_0 W$- and the ${}_0 H$-spaces, respectively. In the two latter cases, restricting to $s \in [0, 2]$, for given $T_0 > 0$ the embeddings hold with a uniform constant for all $0 < T \leq T_0$.
\end{satz}

The embedding \eqref{emb} is a refinement of Proposition 2.11, (2.18), in \protect{\cite{meyries_inter}}. The technical proof can be found in Theorem 2.1.10 in \protect{\cite{butzdiss}}: We use weighted Besov spaces which are another generalization of Sobolev spaces for the scalar valued case, see Subsection 2.1 in \protect{\cite{butzdiss}} for the definition. \\

Moreover, Propositions \ref{unab} enables us to deduce results for multiplication in Slobodetskii spaces. The proof can be found in Lemma 2.1.17 in \protect{\cite{butzdiss}}.
\begin{lem}
	\label{ban-alg} 
	Let $0 < T_0 < \infty$ be fixed and $J= (0, T)$ for $0 < T \leq T_0$. Moreover, let $\mu \in \left(\frac{7}{8},1 \right]$. 
	\begin{enumerate}
		\item{\label{lem1} Let $f \in {}_0 W^{\nicefrac{5}{8}}_{2, \mu} (J)$ and $g \in W^{\nicefrac{1}{8}}_{2, \mu} (J)$, then $fg \in W^{\nicefrac{1}{8}}_{2, \mu} (J)$ and 
			\begin{align*}
			\|f g\|_{W^{\nicefrac{1}{8}}_{2, \mu} (J)} \leq C(T)\|f\|_{W^{\nicefrac{5}{8}}_{2, \mu} (J)} \|g\|_{W^{\nicefrac{1}{8}}_{2, \mu} (J)},
			\end{align*}
			for a constant $C(T) \rightarrow 0$ monotonically as $T \rightarrow 0$.}
		\item{\label{lem2} Let $f, g \in {}_0 W^{\nicefrac{5}{8}}_{2, \mu} (J)$, then $fg \in {}_0 W^{\nicefrac{5}{8}}_{2, \mu} (J)$ and 
			\begin{align*}
			\|f g\|_{W^{\nicefrac{5}{8}}_{2, \mu} (J)} \leq C(T) \|f\|_{W^{\nicefrac{5}{8}}_{2, \mu} (J)} \|g\|_{W^{\nicefrac{5}{8}}_{2, \mu} (J)},
			\end{align*}
			for a constant $C(T) \rightarrow 0$ monotonically as $T \rightarrow 0$, i.e.\ the space ${}_0 W^{\nicefrac{5}{8}}_{2, \mu} (J)$ is a Banach algebra up to a constant in the norm estimate for the product.}
		\item{\label{lem2a} Let $f, g \in W^{\nicefrac{3}{8}}_{2, \mu} (J)$, then $fg \in W^{\nicefrac{1}{8}}_{2, \mu} (J)$ and 
			\begin{align*}
			\|f g\|_{W^{\nicefrac{1}{8}}_{2, \mu} (J)} \leq C(T) \|f\|_{W^{\nicefrac{3}{8}}_{2, \mu} (J)} \|g\|_{W^{\nicefrac{3}{8}}_{2, \mu} (J)},
			\end{align*}
			for a constant $C(T) \rightarrow 0$ monotonically as $T \rightarrow 0$.}
		\item{\label{lem3} Let $f \in W^{s}_{2, \mu} (J)$, $1 > s > \frac{1}{2} - \mu$ such that there exists a $\tilde{C} > 0$ with $|f| \geq \tilde{C}$. Then $\tfrac{1}{f} \in W^{s}_{2, \mu} (J)$ with
			\begin{align*}
			\left\|\tfrac{1}{f} \right\|_{W^{s}_{2, \mu} (J)} \leq C \left(\|f\|_{W^{s}_{2, \mu} (J)}, \tilde{C}, T_0 \right).
			\end{align*}}
	\end{enumerate}
\end{lem}

\begin{rem} \label{banrem}
	Using Lemma \ref{ban-alg}.\ref{lem1} and \ref{ban-alg}.\ref{lem2}, we can state similar claims for functions which do not have trace zero, see Remark 2.1.18 in \protect{\cite{butzdiss}}: Let $0 < T_0 < \infty$ be fixed and $J= (0, T)$ for $0 < T \leq T_0$.
	Furthermore, let $f \in W^{\nicefrac{5}{8}}_{2, \mu} (J)$ and $g \in W^{\nicefrac{1}{8}}_{2, \mu} (J)$, then $fg \in W^{\nicefrac{1}{8}}_{2, \mu} (J)$ and 
	\begin{align*}
	\|f g\|^2_{W^{\nicefrac{1}{8}}_{2, \mu} (J)} \leq C(T_0)\|f\|^2_{W^{\nicefrac{5}{8}}_{2, \mu} (J)} \|g\|^2_{W^{\nicefrac{1}{8}}_{2, \mu} (J)} + C(T_0) |f(0)|^2 \| g \|^2_{W^{\nicefrac{1}{8}}_{2, \mu} (J)},
	\end{align*}
	for a uniform constant ${C}(T_0)$ for all $0 < T \leq T_0$.
\end{rem}

\subsection{Maximal Regularity Spaces with Temporal Weights and Related Embeddings}

An important tool to solve linear problems with optimal regularity is Theorem 2.1 in \protect{\cite{meyries}} on maximal $L_p$-regularity with temporal weights. The following statements are based on the results in \protect{\cite{meyries_inter}}, \protect{\cite{meyries}}, and \protect{\cite{meydiss}}. In the following, we use the notation
\begin{align}
\mathbb{E}_{\mu, T, E} &:= W^1_{2, \mu} \big(J; L_{2} (I; E)\big) \cap L_{2, \mu} \big(J; W^{2m}_{2} (I; E)\big) , \nonumber \\
\mathbb{E}_{0, \mu, E} &:= L_{2, \mu} \big(J; L_{2} (I; E)\big) , \nonumber \\
X_{\mu, E} &:= W_2^{2m\left(\mu - \nicefrac{1}{2}\right)} (I; E), \nonumber \\
\mathbb{F}_{j, \mu, E} &:= W^{\omega_j}_{2, \mu} \big(J; L_{2} (\partial I; E) \big) \cap L_{2, \mu} \big(J; W^{2m \omega_j}_{2} (\partial I; E)\big),
\label{nota}
\end{align}
where $J = (0, T)$ and $\omega_j := 1 - \nicefrac{m_j}{2m} - \nicefrac{1}{4m}$, $j = 1, \dots, m$. We set for convenience
\begin{align*}
\tilde{\mathbb{F}}_{\mu, E} :=  \mathbb{F}_{1, \mu, E} \times \cdots \times \mathbb{F}_{m, \mu, E}.
\end{align*}
All the spaces are equipped with their natural norms. We will omit the subscript $\cdot_E$ in the spaces if $E = \mathbb{R}$, e.g.\ $\mathbb{E}_{\mu, T} := \mathbb{E}_{\mu, T, \mathbb{R}}$. \\

In the following, we want to give some useful embeddings for the involved spaces based on results in \protect{\cite{meyries_inter}}, see Subsections 2.2.1 and 2.2.2 in \protect{\cite{butzdiss}} for the proofs.
\begin{lem}
	\label{unabneu}
	Let $T_0$ be fixed, $J = (0, T)$, $0 < T \leq T_0$, and $I$ a bounded open interval. Let $E$ be a Banach space.
	\begin{enumerate}
		\item{\label{unab1}
			Then 
			\begin{align*}
			\mathbb{E}_{\mu, T, E} \hookrightarrow BUC \big(\bar{J}, X_{\mu, E} \big)
			\end{align*}
			with the estimate
			\begin{align*}
			\|\rho\|_{BUC (\bar{J}, X_{\mu, E})} \leq C(T_0) \left(\|\rho\|_{\mathbb{E}_{\mu, T, E}} + \|\rho_{| t = 0}\|_{X_{\mu, E}} \right). 
			\end{align*}}
		\item{\label{unab1a} Let $m=2$ in \eqref{nota} and $\mu \in \left(\frac{7}{8}, 1 \right]$. Then there exists an $\bar{\alpha} \in (0, 1)$ such that 
			\begin{align*}
			\mathbb{E}_{\mu, T, \mathbb{R}^n} \hookrightarrow C^{\bar{\alpha}}\left(\bar{J}; C^1 (\bar{I}; \mathbb{R}^n) \right)
			\end{align*}
			with the estimate
			\begin{align*}
			\|\rho\|_{C^{\bar{\alpha}}([0, T]; C^1(\bar{I}; \mathbb{R}^n))} \leq C(T_0) \left(\|\rho\|_{\mathbb{E}_{\mu, T, \mathbb{R}^n}} + \|\rho_{| t = 0}\|_{X_{\mu, \mathbb{R}^n}} \right).
			\end{align*}}
		\item{\label{unab234} Let $m=2$ in \eqref{nota} and $\mu \in \left(\frac{7}{8}, 1 \right]$. Then the pointwise realization of the $k$-th spatial derivative $\partial^k_{\sigma}$, $k = 1, \dots, 4$, is a continuous map 
			\begin{align*}
			\mathbb{E}_{\mu, T, E} \hookrightarrow H^{\nicefrac{(4 - k)}{4}}_{2, \mu} \big(J; L_2(I; E) \big) \cap L_{2, \mu}\big(J; H^{4 - k}_{2}(I; E) \big)
			\end{align*}
			with the estimate 
			\begin{align*}
			\left\|\partial^k_{\sigma} \rho \right\|_{H^{\nicefrac{(4 - k)}{4}}_{2, \mu}(J; L_2(I; E)) \cap L_{2, \mu}(J; H^{4 - k}_{2}(I; E))} \leq C(T_0) \left(\|\rho\|_{\mathbb{E}_{\mu, T, E}}  + \|\rho_0\|_{{X_{\mu, E}}} \right).
			\end{align*}}
		\item{\label{unab5678}
			Let $m=2$ in \eqref{nota} and $\mu \in \left(\frac{7}{8}, 1 \right]$. Then the spatial trace operator applied to the $k$-th spatial derivative $tr_{|I} \partial^{k}_\sigma$, $k = 0, 1, 2, 3$, is a continuous map
			\begin{align*}
			\mathbb{E}_{\mu, T, E} \rightarrow W^{\nicefrac{(8 - 2k - 1)}{8}}_{2, \mu}\big(J; L_2(\partial I; E)\big) \cap L_{2, \mu}\big(J; W^{\nicefrac{(8 - 2k - 1)}{2}}_{2}(\partial I; E)\big)
			\end{align*}
			with the estimate 
			\begin{align*}
			\big\|tr_{|I} &\partial^{k}_\sigma \rho \big\|_{W^{\nicefrac{(8 - 2k - 1)}{8}}_{2, \mu}(J; L_2(\partial I; E)) \cap L_{2, \mu}(J; W^{\nicefrac{(8 - 2k - 1)}{2}}_{2}(\partial I; E))} \\ 
			&\hspace{5cm} \leq C(T_0) \left(\|\rho\|_{\mathbb{E}_{\mu, T, E}} + \|\rho_{| t = 0}\|_{{X_{\mu, E}}} \right).
			\end{align*}}
	\end{enumerate}
\end{lem}

The following proposition is used to estimate the non-linearities.
\begin{prop} \label{embed}
	Let $J = (0, T)$ let $I$ be a bounded open interval. Further let $\mu \in \left(\frac{1}{2}, 1 \right]$, $k \in \mathbb{N}$ and $q \in [2, \infty]$. Then 
	\begin{align*} 
	L_{\infty}\left(J, W_{2}^{4(\mu - \nicefrac{1}{2})} (I)\right) \cap L_{2 , \mu} \left(J, W^4_2 (I)\right) \hookrightarrow L_{l , \tilde{\mu}} \left(J, W^{k}_{q} (I)\right)
	\end{align*}
	for $\tilde{\mu} = \mu + (1 - \theta)(1 - \mu) \in [\mu, 1]$,
	if $k + \frac{1}{2} - \frac{1}{q} = 4\left(\mu - \frac{1}{2}\right) (1 - \theta) + 4 \theta$ and $l = \frac{2}{\theta}$ for a $\theta \in (0, 1)$. \\
	The operator norm does not depend on $T$.
\end{prop}

\begin{proof}
	To show the embedding, we use 
	\begin{align}
	W_{2}^{s} (I) \hookrightarrow W^{k}_{q} (I) && \textrm{ for } s - \frac{1}{2} \geq k - \frac{1}{q} \textrm{ with } q \geq 2,
	\label{besov}
	\end{align}
	which follows by the definition of the spaces and by a standard embedding theorem for Besov spaces, see for example Theorem 6.5.1 in \protect{\cite{bergh}}.
	Furthermore, by Theorem 1 in Section 4.3.1 of \protect{\cite{triebel}}, we have
	\begin{align*}
	W_{2}^{s} (I) = \left(W_{2}^{4(\mu - \nicefrac{1}{2})} (I), W^4_2 (I)\right)_{\theta, 2} &&\textrm{ for } s = 4 \left(\mu - \frac{1}{2}\right) (1 - \theta) + 4 \theta, \theta \in (0, 1)
	\end{align*}
	with
	\begin{align}
	\| \phi \|_{W_{2}^{s} (I)} \leq C \| \phi \|_{W_{2}^{4(\mu - \nicefrac{1}{2})}(I)}^{1 - \theta} \| \phi \|_{W^4_2 (I)}^{\theta} && \textrm{ for } \phi \in W_{2}^{4(\mu - \nicefrac{1}{2})}(I) \cap W^4_2 (I) = W^4_2 (I).
	\label{inter1}
	\end{align}
	
	In the following, we consider $\|\rho\|_{L_{l , \tilde{\mu}} (J, W^{k}_{q} (I))}$ for $\rho \in L_{\infty} (J, W_{2}^{4(\mu - \nicefrac{1}{2})} (I)) \cap L_{2 , \mu} (J, W^4_2 (I))$ for $l = \nicefrac{2}{\theta}$ and $\tilde{\mu} = \mu + (1 - \theta)(1 - \mu)$. To this end, we use $s = k + \nicefrac{1}{2} - \nicefrac{1}{q} = 4 (\mu - \nicefrac{1}{2}) (1 - \theta) + 4 \theta$ and combine \eqref{besov} and \eqref{inter1}. We obtain 
	\begin{align*}
	\|\rho\|_{L_{l , \tilde{\mu}} (J, W^{k}_{q} (I))} \leq C \|\rho\|_{L_{l , \tilde{\mu}} (J, W_{2}^{s} (I))} \leq  C \left\| \| \rho (t) \|_{W_{2}^{4(\mu - \nicefrac{1}{2})}(I)}^{1 - \theta} \| \rho (t) \|_{ W^4_2 (I)}^{\theta} \right\|_{L_{l , \tilde{\mu}} (J)}.
	\end{align*}
	Taking care of the time weight, we use ${1-\tilde{\mu}} = \theta (1 - \mu)$ to deduce
	\begin{align*}
	\|\rho\|_{L_{l , \tilde{\mu}} (J, W^{k}_{q} (I))} \leq
	C \left\| \| \rho (t) \|_{W_{2}^{4(\mu - \nicefrac{1}{2})}(I)}^{1 - \theta} \left(t^{1-{\mu}} \| \rho (t) \|_{ W^4_2 (I)} \right)^{\theta} \right\|_{L_{l} (J)}.
	\end{align*}
	By Hölder's inequality for $\tilde{p} = \nicefrac{\infty}{1 - \theta}$ and $\tilde{q} = \nicefrac{2}{\theta}$,
	it follows
	\begin{align*}
	\|\rho\|_{L_{l , \tilde{\mu}} (J, W^{k}_{q} (I))} &\leq C \left\| \| \rho (t) \|_{W_{2}^{4(\mu - \nicefrac{1}{2})}(I)} \right\|_{{L_{\infty}(J)}}^{1 - \theta}  \left\| t^{1- \mu} \| \rho (t) \|_{ W^4_2 (I)} \right\|_{L_{^2}(J)}^{\theta} \\
	&\leq C \left(\| \rho \|_{L_{\infty}(J, W_{2}^{4(\mu - \nicefrac{1}{2})}(I))} + \| \rho \|_{L_{2, \mu}(J, W^4_2 (I))}\right),
	\end{align*}
	where Young's inequality yields the last estimate. The constant does not depend on $T$.
\end{proof}

\subsection{An Estimate for the Reciprocal Length by the Curvature}
 
\begin{lem} \label{kappabound}
	Let $\alpha \in (0, \pi)$. Furthermore, let $c: [0, 1] \rightarrow \mathbb{R}^2, \sigma \mapsto c(\sigma),$ be a regular curve of class $C^2$ parametrized proportional to arc length. Moreover, let the unit tangent $\tau := \frac{\partial_\sigma c}{\mathcal{L}[c]}$ fulfill
	\begin{align*}
	\tau (\sigma) = \begin{pmatrix} \cos \alpha \\ \pm \sin \alpha \end{pmatrix} && \text{ for } \sigma = 0, 1,
	\end{align*}
	where $\mathcal{L}[c]$ denotes the length of the curve.
	Then it holds
	\begin{align*}
	\frac{1}{\mathcal{L}[c]} \leq \frac{1}{\sqrt{2}\sin \alpha} \left\|\kappa[c] \right\|_{C([0, 1])}.
	\end{align*}
\end{lem}

\begin{proof}
	We denote by $\vec{\kappa} = \nicefrac{\partial^2_\sigma c}{(\mathcal{L}[c])^2}$ and $\kappa = \langle \nicefrac{\partial^2_\sigma c}{(\mathcal{L}[c])^2}, R \tau \rangle$ the curvature vector and the scalar curvature, respectively. Here, $\langle \cdot, \cdot \rangle$ denotes the euclidean inner product on $\mathbb{R}^2$ and $R$ the matrix which rotates vectors counterclockwise in $\mathbb{R}^2$ by the angle $\nicefrac{\pi}{2}$. We deduce 
	\begin{align*}
	|\tau(1) - \tau(0)| = \left|\int_0^1 \partial_\sigma \tau(x) \dd x \right| = \left|\int_0^1 \mathcal{L}[c] \vec{\kappa}[c](x) \dd x \right| \leq \mathcal{L}[c] \|\kappa[c]\|_{C([0, 1])},
	\end{align*}
	where we used that $\langle \vec{\kappa}, \tau \rangle = 0$. 
	Moreover, we have
	\begin{align*}
	|\tau(1) - \tau(0)|^2 = 1 - 2 \langle \tau(1), \tau(0) \rangle + 1 = 2 - 2 (\cos \alpha)^2 + 2 (\sin \alpha)^2 \geq 2 (\sin \alpha)^2 > 0, 
	\end{align*}
	for $\alpha \in (0, \pi)$. Combining both estimates, we deduce the claim.
\end{proof}

\section{Reduction of the Geometric Evolution Equation to a Partial Differential Equation} \label{rog}
In order to reduce the geometric evolution equation to a partial differential equation on a fixed interval, we employ a parametrization which is similar to the one established in \protect{\cite{vogel}}: Let $\Phi^*: [0, 1] \rightarrow \mathbb{R}^2$ be a regular $C^5$-curve parametrized proportional to arc length, such that $\Lambda := \Phi^* ([0,1])$ fulfill the conditions
\begin{align}
\Phi^*(\sigma) &\in \mathbb{R} \times  \{0\} && \text{ for } \sigma \in \{0, 1\}, \nonumber \\
\measuredangle \left({n}_{\Lambda}(\sigma),  \begin{pmatrix} 0 \\ -1\end{pmatrix} \right) &= \pi - \alpha && \text{ for } \sigma \in \{0, 1\}, \nonumber \\
{\kappa}_{\Lambda} (\sigma) &= 0 && \text{ for } \sigma \in \{0, 1\},
\label{bound1}
\end{align}
where $\tau_{\Lambda}(\sigma) := \nicefrac{\partial_\sigma \Phi^* (\sigma)}{\mathcal{L}[\Phi^* ]}$
and $n_{\Lambda}(\sigma) := R\tau_{\Lambda}(\sigma)$ are the unit tangent and unit normal vector of $\Lambda$ at the point $\Phi^*(\sigma)$ for $\sigma \in [0, 1]$, respectively. Here, $R$ denotes the counterclockwise $\nicefrac{\pi}{2}$-rotation matrix. Furthermore, the curvature vector of $\Lambda$ at $\Phi^*(\sigma)$ is given by $\vec{\kappa}_{\Lambda}(\sigma) := \nicefrac{\partial^2_\sigma \Phi^* (\sigma)}{(\mathcal{L}[\Phi^* ])^2}$ for $\sigma \in [0, 1]$. \\

For a sufficiently small $d$, curvilinear coordinates are defined as
\begin{align}
\Psi: [0, 1] \times (- d, d) & \rightarrow \mathbb{R}^2 \nonumber \\
(\sigma, q) & \mapsto \Phi^*(\sigma) + q \left(n_{\Lambda}(\sigma) + {\cot{\alpha}} \eta(\sigma) \tau_{\Lambda} (\sigma) \right), \label{curvilin}
\end{align}
where the function $\eta: [0, 1] \rightarrow [-1, 1]$ is given by
\begin{align}
\eta(x) := \begin{cases}
-1 & \textrm{ for } 0\leq x < \frac{1}{6}\\
0 & \textrm{ for } \frac{2}{6} \leq x < \frac{4}{6}\\
1 & \textrm{ for } \frac{5}{6} \leq x \leq 1 \\
\textrm{arbitrary} & \textrm{ else},
\end{cases}
\label{eta}
\end{align}
such that it is monotonically increasing and smooth. Note, that if $\alpha = \nicefrac{\pi}{2}$, then $\cot \alpha =0$ and the second summand in the definition of $\Psi$ vanishes, cf.\ \eqref{curvilin}. 

\begin{rem} \label{smooth}
	\begin{enumerate}
		\item \label{linf} It is trivial that $q \mapsto \Psi(\sigma, q)$ is smooth for $\sigma \in [0, 1]$. Since $\Phi^*$ and $\eta$ are functions of class $C^5$, it follows that $\Psi \in C^4([0, 1] \times (-d, d))$ and
		\begin{align*}
		\|\Psi\|_{C^4([0, 1] \times (-d, d))} \leq C(\alpha, \Phi^*, \eta, |d|).
		\end{align*}
		\item The tangential part is weighted by the function $\eta$, which assures that $[\Psi(\sigma, q)]_2 = 0$ for $\sigma \in \{0, 1\}$ and each $q \in (-d, d)$. This is important, since we want the solution to have its boundary points on the real axis, cf.\ \eqref{g11}. The tangential part is constant in a neighborhood of the boundary points and it vanishes in the middle of the curve.
	\end{enumerate}
\end{rem}

In the following, we consider functions
\begin{align*}
\rho: [0,T) \times [0, 1] \rightarrow (- d, d), \;\; (t, \sigma) \mapsto \rho(t, \sigma)
\end{align*}
and define 
\begin{align}
\Phi(t, \sigma) := \Psi(\sigma, \rho(t, \sigma)). \label{kurve}
\end{align}
An evolving curve is now given by 
\begin{align}
\Gamma_t := \{\Phi(t, \sigma) | \; \sigma \in [0, 1] \},
\label{par}
\end{align}
where we obtain $\Phi(t, \sigma) \in \mathbb{R} \times \{0\}$ for $\sigma \in \{0, 1\}$, $t \in [0, T)$, i.e.\ \eqref{g11}, by construction. \\ 

Next, we want to express \eqref{g1}--\eqref{g3} with the help of the parametrization induced by \eqref{par}. 
In order to improve readability, we omit the arguments at some points, e.g.\ $\Psi_\sigma = \Psi_\sigma(\sigma, \rho(t, \sigma))$ and $\rho = \rho(t, \sigma)$. Assuming $|\Phi_{\sigma}(t, \sigma)| \neq 0$, we derive for the arc length parameter $s$ of $\Gamma_t$
\begin{align*}
\frac{ds}{d\sigma} = |\Phi_\sigma| = \sqrt{|\Psi_\sigma|^2 + 2\langle \Psi_\sigma, \Psi_q \rangle \partial_{\sigma} \rho + |\Psi_q|^2(\partial_{\sigma} \rho)^2} =: J(\rho) = J(\sigma, \rho, \partial_\sigma \rho),
\end{align*}
where $|\cdot|$ and $\langle \cdot, \cdot \rangle$ denote the Euclidean norm and the inner product in $\mathbb{R}^2$, respectively. 
Thus, the unit tangent $\tau_{\Gamma_t}$ and the outer unit normal $n_{\Gamma_t}$ of the curve $\Gamma_t$ are given by
\begin{align*}
\tau_{\Gamma_t} = \frac{1}{J(\rho)} \Phi_{\sigma}= \frac{1}{J(\rho)} (\Psi_\sigma + \Psi_q \partial_{\sigma} \rho), && n_{\Gamma_t} = R\tau_{\Gamma_t} = \frac{1}{J(\rho)} R \Phi_\sigma.
\end{align*}
For the scalar normal velocity $V$ of $\Gamma_t$, we have
\begin{align*}
V &= \langle \Phi_t , n_{\Gamma_t} \rangle = \frac{1}{J(\rho)} \langle \Phi_t , R \Phi_{\sigma}\rangle 
= \frac{1}{J(\rho)} \langle \Psi_q , R \Psi_\sigma \rangle \rho_t.
\end{align*}
Moreover, the Laplace-Beltrami operator on $\Gamma_t$ as a function in $\rho$ is defined by
\begin{align*}
\Delta(\rho) := \partial_s^2 = \frac{1}{J(\rho)} \partial_\sigma \left(\frac{1}{J(\rho)} \partial_\sigma \right) = \frac{1}{J(\rho)} \partial_\sigma \left(\frac{1}{J(\rho)} \right) \partial_\sigma + \frac{1}{(J(\rho))^2}  \partial_\sigma^2.
\end{align*}
Thereby, the scalar curvature of $\Gamma_t$ as function in $\rho$ can be expressed by
\begin{align}
\kappa(\rho) &= \frac{1}{(J(\rho))^3} \langle \Psi_q , R \Psi_\sigma \rangle \partial^2_{\sigma} \rho + U,
\label{kappa}
\end{align}
where $U = U(\sigma, \rho, \partial_{\sigma} \rho)$ denotes terms of the form
\begin{align}
U(\sigma, \rho, \partial_{\sigma} \rho) = C \left(J(\rho)\right)^k \left(\prod_{i = 0}^{p} \left\langle \Psi^{\beta_i}_{(\sigma, q)} , R \Psi^{\gamma_i}_{(\sigma, q)} \right\rangle (\sigma, \rho) \right) (\partial_{\sigma} \rho)^r,
\label{u}
\end{align}
with $C \in \mathbb{R}$, $k \in \mathbb{Z}$, $p, r \in \mathbb{N}_0$, and $\beta_i, \gamma_i \in \mathbb{N}_0^2$, such that $|\beta_i|, |\gamma_i| \geq 1$ and $|\beta_i| + |\gamma_i| \leq 3$ for all $i \in \{0, \dots, p\}$, cf.\ Appendix A.1 of \protect{\cite{butzdiss}} for more details on this derivation and the following ones. 
Consequently, the first derivative of the curvature is given by
\begin{align}
\partial_{s} \kappa(\rho) &= \frac{1}{(J(\rho))^4} \langle \Psi_q , R \Psi_\sigma \rangle \partial^3_{\sigma} \rho 
+ T \left(\partial^2_{\sigma} \rho \right)^2 + T \partial^2_{\sigma} \rho + T,
\label{ks}
\end{align}
where the prefactors $T = T(\sigma, \rho, \partial_{\sigma} \rho)$ denote terms of the form
\begin{align}
T(\sigma, \rho, \partial_{\sigma} \rho) = C \left(J(\rho)\right)^k \left(\prod_{i = 0}^{p} \left\langle \Psi^{\beta_i}_{(\sigma, q)} , R \Psi^{\gamma_i}_{(\sigma, q)} \right\rangle (\sigma, \rho) \right) (\partial_{\sigma} \rho)^r,
\label{t}
\end{align}
with $C \in \mathbb{R}$, $k \in \mathbb{Z}$, $p, r \in \mathbb{N}_0$, and $\beta_i, \gamma_i \in \mathbb{N}_0^2$, such that $|\beta_i|, |\gamma_i| \geq 1$ and $|\beta_i| + |\gamma_i| \leq 4$ for all $i \in \{0, \dots, p\}$. Moreover, we have
\begin{align}
\partial^2_{s} \kappa(\rho) &= \frac{1}{(J(\rho))^5} \langle \Psi_q , R \Psi_\sigma \rangle \partial^4_{\sigma} \rho 
+ \tilde{S} \partial^3_{\sigma} \rho \partial^2_{\sigma} \rho 
+ \tilde{S} \partial^3_{\sigma} \rho 
+ \tilde{S} \left(\partial^2_{\sigma} \rho \right)^3 + \tilde{S} \left(\partial^2_{\sigma} \rho \right)^2 + \tilde{S} \partial^2_{\sigma} \rho + \tilde{S},  \label{kss}
\end{align}
with the prefactors of the form
\begin{align}
\tilde{S}(\sigma, \rho, \partial_{\sigma} \rho) := C \left(J(\rho)\right)^k \left(\prod_{i = 0}^{p} \left\langle \Psi^{\beta_i}_{(\sigma, q)} , R \Psi^{\gamma_i}_{(\sigma, q)} \right\rangle (\sigma, \rho) \right) (\partial_{\sigma} \rho)^r,
\label{tildes}
\end{align}
with $C \in \mathbb{R}$, $k \in \mathbb{Z}$, $p, r \in \mathbb{N}_0$, and $\beta_i, \gamma_i \in \mathbb{N}_0^2$, such that $|\beta_i|, |\gamma_i| \geq 1$ and $|\beta_i| + |\gamma_i| \leq 5$ for all $i \in \{0, \dots, p\}$.
We assume that $\nicefrac{1}{J(\rho)} \langle \Psi_q , R \Psi_\sigma \rangle \neq 0$ and obtain by \eqref{g1} the equation
\begin{align}
\rho_t = - \frac{J(\rho)}{\langle \Psi_q, R \Psi_\sigma \rangle} \Delta (\rho) \kappa(\rho) && \textrm{ for } \sigma \in (0, 1) \textrm{ and } t > 0.
\label{divide}
\end{align}
Furthermore, the boundary condition \eqref{g2} is represented by 
\begin{align}
\cos (\pi - \alpha) &= \left\langle n_{\Gamma_t}, \begin{pmatrix} 0 \\ -1\end{pmatrix} \right\rangle = \left\langle \tau, R^T \begin{pmatrix} 0 \\ -1\end{pmatrix} \right\rangle 
= \frac{1}{J(\rho)} \left\langle (\Psi_\sigma + \Psi_q \partial_{\sigma} \rho), \begin{pmatrix} -1 \\ 0 \end{pmatrix} \right\rangle. 
\label{schlecht}
\end{align}
Straightforward calculations together with the assumptions on the reference curve $\Phi^*$ show that this is fulfilled if and only if
\begin{align}
\partial_{\sigma} \rho (t, \sigma) = 0 && \textrm{ for } \sigma  \in \{0, 1\} \textrm{ and } t > 0.
\label{anglecon} 
\end{align}
This gives the reformulation of the angle condition.\\

In summary, we have deduced that the problem \eqref{g1}-\eqref{g3} can be expressed by the partial differential equation given in \eqref{prob}, if $\Phi_\sigma(\sigma, t) \neq 0$ holds true. To this end, we introduce the following lemma:
\begin{lem} \label{regu}
	Let $\rho: [0, 1] \rightarrow \mathbb{R}$ satisfy the bound
	\begin{align}
	\|\rho \|_{C([0, 1])} &< \frac{1}{2 \|\kappa_{\Lambda}\|_{C(\bar{I})} \left( 1 + (\cot \alpha)^2 + \hat{C}|\cot{\alpha}| \|\eta' \|_{C([0, 1])}  \right)} =: K_0(\alpha, \Phi^*, \eta),
	\label{kleinrho} \\
	\intertext{and in the case $\alpha \neq \frac{\pi}{2}$ additionally}
	\|\partial_\sigma \rho \|_{C([0, 1])} &< \frac{\mathcal{L}[\Phi^*]}{12 |\cot{\alpha}|} =: K_1(\alpha, \Phi^*), \label{kleindrho}
	\end{align}
	where $\hat{C} := \sqrt{2}\sin \alpha > 0$. Then $J(\rho) > 0$ and $[0, 1] \ni \sigma \mapsto \Psi(\sigma, \rho(\sigma))$ is a regular parametrization.\\ 
	In particular, if \eqref{kleinrho} and \eqref{kleindrho} are fulfilled for $\tfrac{2}{3}K_0$ and $\tfrac{2}{3}K_1$, respectively, then there exists a $C(\alpha, \Phi^*, \eta) > 0$, such that 
	\begin{align}
	J(\rho) > C(\alpha, \Phi^*, \eta) > 0.
	\label{jrho}
	\end{align}
\end{lem}

\begin{proof}
Using the identities
\begin{align*}
\partial_{\sigma} \Phi^* = \mathcal{L}[\Phi^* ] \tau_{\Lambda},  && \partial_{\sigma} \tau_{\Lambda}  = \mathcal{L}[\Phi^* ] \kappa n_{\Lambda}, && \partial_{\sigma} n_{\Lambda} = - \mathcal{L}[\Phi^* ] \kappa \tau_{\Lambda},
\end{align*}	
we see by direct estimates that $|\Phi_{\sigma}(\sigma)|^2 > 0$, see Lemma 5.1.2 in \protect{\cite{butzdiss}} for details.  
\end{proof}

We give a sketch of the representation in Figure \ref{4}.
\begin{center}\vspace{0 cm}
	\scalebox{1}{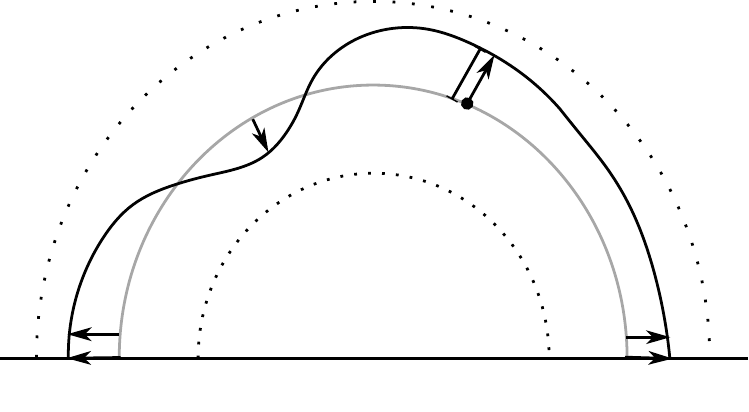}
	\captionof{figure}{Representation of a curve by curvilinear coordinates and a height function.}
	\label{4}
\end{center}\vspace{0 cm}

In the following, we use the notation from \eqref{nota} with $m = 2$. By construction, it follows:
\begin{rem} \label{initial}
	\begin{enumerate}
		\item \label{initial1} Let $\rho_0$ fulfill the assumptions of Theorem \ref{local}. Then $\rho_0 \in X_{\mu}$, $\mu \in (\nicefrac78, 1 ]$, implies $\rho_0 \in C^{1}\big(\bar{I}\big)$ by Remark \ref{afterunab},
		as $4(\mu - \nicefrac{1}{2}) - \nicefrac{1}{2} > 1$. It follows by Lemma \ref{regu} that 
		$J(\rho_0) > C(\alpha, \Phi^*, \eta) > 0$ and that
		\begin{align*}
		[0, 1] \ni \sigma \mapsto \Phi(0, \sigma) = \Psi(\sigma, \rho_0(\sigma))
		\end{align*} 
		is a regular parametrization of the initial curve corresponding to $\rho_0$.
		It holds by construction that $[\Phi(0, \sigma)]_2 = [\Psi(\sigma, \rho_0(\sigma))]_2 = 0$ for $\sigma \in \{0, 1\}$.
		\item Note that by \eqref{anglecon} the compatibility condition \eqref{komp} implies
		\begin{align*}
		\measuredangle \left({n}_{\Gamma_0}(\sigma),  \begin{pmatrix} 0 \\ -1\end{pmatrix} \right) &= \pi - \alpha && \text{ for } \sigma \in \{0, 1\}, \nonumber
		\end{align*}
		where $\Gamma_0:= \Phi(0, \bar{I})$. Thus, the $\alpha$-angle condition holds true for the initial curve. 
	\end{enumerate}
\end{rem}

Using the previous considerations, we are able to reduce the geometric evolution equation to a quasilinear partial differential equation on a fixed interval: To this end, the representations derived in \eqref{kss}, \eqref{ks}, and \eqref{anglecon} are used and plugged into \eqref{prob}. Putting the highest order terms on the left-hand side and remaining ones on the right-hand side, we formally obtain that problem \eqref{prob} is an initial value problem for a quasilinear parabolic partial differential equation of the form
\begin{align}
\rho_t + a(\sigma, \rho, \partial_\sigma \rho) \partial_\sigma^4 \rho &= f(\rho, \partial_\sigma \rho, \partial_\sigma^2 \rho, \partial_\sigma^3 \rho) && \textrm{ for } (t, \sigma) \in (0, T) \times [0, 1],\nonumber \\
b_1(\sigma) \partial_\sigma \rho&= - g_1 && \textrm{ for } \sigma \in \{0, 1\} \textrm{ and } t \in (0, T),\nonumber  \\
b_2(\sigma, \rho, \partial_\sigma \rho) \partial_\sigma^3 \rho &= - g_2(\rho, \partial_\sigma \rho, \partial^2_\sigma \rho) && \textrm{ for } \sigma \in \{0, 1\} \textrm{ and } t \in (0, T) ,\nonumber  \\
\rho |_{t=0} &= \rho_0 &&  \textrm{ in } [0, 1].
\label{pde}
\end{align}
Here, the coefficients on the left-hand side are given by
\begin{align}
a(\sigma, \rho, \partial_\sigma \rho) : = \frac{1}{(J(\rho))^4}, &&
b_1(\sigma) := 1, &&  b_2(\sigma, \rho, \partial_\sigma \rho) := \frac{1}{(J(\rho))^4} \langle \Psi_q , R \Psi_\sigma \rangle.
\label{coeffs}
\end{align}
The right-hand side of the first equation is defined by
\begin{align}
f(\sigma, \rho, \partial_\sigma \rho, \partial_\sigma^2 \rho, \partial_\sigma^3 \rho) &:=
S \partial^3_{\sigma} \rho \partial^2_{\sigma} \rho 
+ S \partial^3_{\sigma} \rho + S \left(\partial^2_{\sigma} \rho \right)^3 + S \left(\partial^2_{\sigma} \rho \right)^2 + S \partial^2_{\sigma} \rho + S,
\label{darf}
\end{align}
where the prefactors $S = S(\sigma, \rho, \partial_{\sigma} \rho)$ denote terms of the form
\begin{align}
S(\sigma, \rho, \partial_{\sigma} \rho) := \frac{- J(\rho)}{\langle \Psi_q , R \Psi_\sigma \rangle (\sigma, \rho) } \tilde{S}(\sigma, \rho, \partial_{\sigma} \rho) && \textrm{ for } \tilde{S} \textrm{ of the form } \eqref{tildes}.
\label{s}
\end{align}
Moreover, the right-hand sides of the second and third equation are given by
\begin{align}
g_1(\sigma) := 0, &&
g_2(\sigma, \rho, \partial_\sigma \rho, \partial^2_\sigma \rho ) := T \left(\partial^2_{\sigma} \rho \right)^2 + T \partial^2_{\sigma} \rho + T,
\label{darg}
\end{align}
where the prefactors $T=T(\sigma, \rho, \partial_\sigma \rho)$ are introduced in \eqref{t}. \\

Deriving the motion law, we assumed $\nicefrac{1}{J(\rho)} \langle \Psi_q , R \Psi_\sigma \rangle(\sigma, \rho) \neq 0$, see \eqref{divide}. As confirmation, the following lemma is proven:
\begin{lem} \label{regu1}
	Let $\rho_0 \in X_{\mu}$, such that the condition \eqref{small} is satisfied. Furthermore, let $K \in \mathbb{R}^+$.
	Then, there exists a $\tilde{T} = \tilde{T}(\alpha, \Phi^*, \eta, K, R)$  with $\|\rho_0\|_{X_{\mu}} \leq R$, such that for $\rho \in \mathbb{E}_{\mu, T}$ fulfilling $\rho_{|t=0} = \rho_0$ and $\|\rho\|_{\mathbb{E}_{\mu, T}} \leq K$ for $0 < T < \tilde{T}$ it holds: $\rho$ fulfills the bounds \eqref{kleinrho} and \eqref{kleindrho} 
	for $\tfrac{2}{3}K_0$ and $\tfrac{2}{3}K_1$. In particular,
	\begin{align}
	J(\rho) > C(\alpha, \Phi^*, \eta) > 0 && \textrm{ for } \sigma \in [0, 1] \textrm{ and all } 0 \leq t \leq T \textrm{ with } 0 < T < \tilde{T},
	\label{klein}
	\end{align}
	where $C(\alpha, \Phi^*, \eta)$ is given by \eqref{jrho}. Additionally, $[0, 1] \ni \sigma \mapsto \Psi(\sigma, \rho(t, \sigma))$ is a regular parametri\-zation for all $0 \leq t \leq T$ with $0 < T < \tilde{T}$. 
\end{lem}

\begin{proof}
	By the embedding 
	\begin{align*}
	\mathbb{E}_{\mu, T} \hookrightarrow C^{\bar{\alpha}}\left(\bar{J}; C^1 (\bar{I}) \right),
	\end{align*}
	see Lemma \ref{unabneu}.\ref{unab1a}, we have 
	\begin{align*}
	\|\partial^i_\sigma \rho(t) - \partial^i_\sigma \rho_0\|_{C(\bar{I})} \leq T^{\bar{\alpha}} C \left(\|\rho\|_{\mathbb{E}_{\mu, T}} + \|\rho_{| t = 0}\|_{X_\mu} \right), 
	\end{align*}
	where $\bar{\alpha} > 0$, $i \in \{0, 1\}$, and $t \in [0, T]$. Note that the constant $C$ does not depend on $T$. Thus, by choosing $\tilde{T}$ small enough, it holds for $i = 0, 1$
	\begin{align*}
	\|\partial^i_\sigma \rho(t) - \partial^i_\sigma \rho_0\|_{C(\bar{I})} \leq \tilde{T}^{\bar{\alpha}} C \left(K + \|\rho_{| t = 0}\|_{X_\mu} \right) < \frac{\min \{K_0, K_1\}}{3}.
	\end{align*}
	Consequently, 
	\begin{align*}
	\|\partial^i_\sigma \rho(t)\|_{C(\bar{I})} \leq \|\partial^i_\sigma \rho(t) - \partial^i_\sigma \rho_0\|_{C(\bar{I})} + \|\partial^i_\sigma \rho_0\|_{C(\bar{I})} < \frac{2K_i}{3}
	\end{align*}
	is obtained for $i \in \{0, 1\}$. The addendum follows directly by \eqref{jrho} in Lemma \ref{regu}.
\end{proof}

\begin{definition} \label{ass}
	For fixed $K$, we consider the corresponding $\tilde{T}(\alpha, \Phi^*, \eta, K, R) > 0$ with $\|\rho_0\|_{X_{\mu}} \leq R$, which is given in Lemma \ref{regu1}. We set 
	\begin{align*}
	\mathcal{B}_{K, T} := \left\{ \rho \in  \mathbb{E}_{\mu, T} : \|\rho\|_{ \mathbb{E}_{\mu, T}} \leq K, \rho_{|t=0} = \rho_0 \right\} && \textrm{ for } 0 < T < \tilde{T}.
	\end{align*}
\end{definition}

Then, we obtain:
\begin{rem} \label{wohldef}
	\begin{enumerate}
		\item \label{wd1} There exists a $C(\alpha, \Phi^*, \eta, K) > 0$, such that
		\begin{align*}
		\langle \Psi_q , R \Psi_\sigma \rangle(\sigma, \rho)  > C(\alpha, \Phi^*, \eta, K) > 0 &&\textrm{ for } \sigma \in [0, 1] \textrm{ and } \rho \in \mathcal{B}_{K, T} \textrm{ with }  0 < T < \tilde{T}.
		\end{align*}
		Moreover, by direct calculations, we obtain
		\begin{align*}
		\langle \Psi_q, R \Psi_\sigma \rangle (\sigma, \rho) 
		&= \mathcal{L}[\Phi^*] \left[1- \rho \left(\kappa_{\Lambda} - \frac{1}{\mathcal{L}[\Phi^*]}\cot{\alpha} \eta' - (\cot{\alpha})^2 \eta^2 \kappa_{\Lambda} \right) \right]. 
		\end{align*}
		Exploiting Lemma \ref{kappabound}, we deduce for $\rho \in \mathcal{B}_{K, T}$
		\begin{align*}
		\langle \Psi_q, R \Psi_\sigma \rangle (\sigma, \rho)
		&\geq \mathcal{L}[\Phi^*] \left[1- |\rho| \; \|\kappa_{\Lambda}\|_{C(\bar{I})} \left( 1 + \hat{C}|\cot{\alpha} \eta'| + (\cot {\alpha})^2 \right) \right] \\
		&> C(\alpha, \Phi^*, \eta, K) > 0.
		\end{align*}
		\item \label{gross} By the expressions for the derivatives of $\Psi$, we additionally derive
		\begin{align*}
		J(\rho) < C(\alpha, \Phi^*, \eta, K) &&\textrm{ for } \sigma \in [0, 1] \textrm{ and } \rho \in \mathcal{B}_{K, T} \textrm{ with }  0 < T < \tilde{T}.
		\end{align*}
		\item \label{wohldef0} The previous statements hold also true if $\rho \in \mathcal{B}_{K, T}$ is replaced by an initial datum $\rho_0$ fulfilling the assumptions of Theorem \ref{local}.
	\end{enumerate}
\end{rem}

In order to solve the partial differential equation \eqref{pde}, it is linearized and the generated linear problem is solved with optimal regularity.

\section{The Linear Problem} \label{line}

In order to solve the problem \eqref{pde}, we linearize it around the initial datum $\rho_0$. In the following, we want to show that, for a suitable choice of the function spaces, the linearized problem
\begin{align}
\rho_t +  \mathcal{A} \rho &= F(t, \sigma) && \textrm{ for } (t, \sigma) \in (0, T) \times [0, 1], \nonumber \\
\mathcal{B}_1 \rho &= G_1 (t, \sigma) && \textrm{ for } \sigma \in \{0, 1\} \textrm{ and } t \in (0, T), \nonumber \\
\mathcal{B}_2 \rho &= G_2 (t, \sigma) && \textrm{ for } \sigma \in \{0, 1\} \textrm{ and } t \in (0, T), \nonumber \\
\rho(0, \sigma) &= \rho_0(\sigma) && \textrm{ for } \sigma \in [0, 1], \label{linprob}
\end{align}
has a unique solution $\rho$ with optimal regularity. The linear operators $\mathcal{A}$, $\mathcal{B}_1$ and $\mathcal{B}_2$ are given by
\begin{align}
\mathcal{A} &:= 
a(\sigma, \rho_0, \partial_\sigma \rho_0) D^{4} && \textrm{ for } \sigma \in [0, 1],
\nonumber \\
\mathcal{B}_1 &:= 
i b_{1}(\sigma) \operatorname{tr}_{\partial I} D^{1} && \textrm{ for } \sigma \in \{0, 1\},
\nonumber \\
\mathcal{B}_2 &:= 
i^3 b_{2}(\sigma, \rho_0, \partial_\sigma \rho_0) \operatorname{tr}_{\partial I} D^{3} && \textrm{ for } \sigma \in \{0, 1\}, 
\label{operator}
\end{align}
where $D = -i \partial_{\sigma}$ and 
$a, b_1$, and $b_2$ are given by \eqref{coeffs}. Moreover, for a fixed but arbitrary $\bar{\rho} \in \mathcal{B}_{K, T}$, cf.\ Definition \ref{ass}, we set
\begin{align*}
F(t, \sigma) &:= - \left(\frac{1}{(J(\bar{\rho}))^4} - \frac{1} {(J(\rho_0))^4} \right) \partial_\sigma^4 \bar{\rho} + f(\bar{\rho}, \partial_\sigma \bar{\rho}, \partial_\sigma^2 \bar{\rho}, \partial_\sigma^3 \bar{\rho}), \\ 
G_1(t, \sigma) &:= 0, \\
G_2(t, \sigma) &:= - \left( b_2(\bar{\rho}, \partial_\sigma \bar{\rho}) - b_2(\rho_0, \partial_\sigma \rho_0) \right) \partial_\sigma^3 \bar{\rho} - g_2(\bar{\rho}, \partial_\sigma \bar{\rho}, \partial_\sigma^2 \bar{\rho}),
\end{align*}
and $\tilde{G} := (G_1, G_2)$. After solving the linear problem \eqref{linprob} for a general $\bar{\rho} \in \mathcal{B}_{K, T}$, we will prove the well-posedness of the partial differential equation \eqref{pde} by a fixed point argument. \\

The existence of a unique solution for the linearized problem \eqref{linprob} will be proven by applying the maximal regularity result stated in Theorem 2.1 of \protect{\cite{meyries}}, see Theorem 2.2.1 in \protect{\cite{butzdiss}} for a simplified version of this theorem, which is adapted to this situation.
Applying this, we deduce the following theorem:
\begin{satz} [Existence for the Linear Problem \eqref{linprob}] \label{linex}
	Let the assumptions of Theorem \ref{local} hold true and let $K$ and $\tilde{T}$ be given by Definition \ref{ass}.
	Then the linearized problem \eqref{linprob} possesses a unique solution $\rho \in \mathbb{E}_{\mu, T}$ for $0 < T < \tilde{T}$, if and only if 
	\begin{align*}
	F \in \mathbb{E}_{0, \mu},\; \tilde{G} \in \tilde{\mathbb{F}}_{\mu},
	\end{align*}
	and the compatibility condition
	\begin{align}
	\mathcal{B}_1 (0, \cdot, D) \rho_0 = G_1 (0, \cdot) = 0 && \textrm{ for } \sigma \in \{0, 1\}
	\label{comp}
	\end{align}
	is fulfilled.
\end{satz}

\begin{rem} \label{comprem}
	Note that the compatibility condition \eqref{comp} is fulfilled immediately if \eqref{komp} holds true.
\end{rem}

\begin{proof} [Proof of Theorem \ref{linex}]
	We have to check the respective assumptions of Theorem 2.1 in \protect{\cite{meyries}} for the operators given in \eqref{operator}, where the order of $\mathcal{A}$ is $4 = 2m$, thus $m = 2$. Moreover, we work with the integration order $p = 2$ and image space $E := \mathbb{C}$, which is a Banach space of class $\mathcal{HT}$, as it is a Hilbert space. \\
	
	All the involved operators are well-defined and non-trivial: First of all, by Remark \ref{initial}.\ref{initial1} a lower bound on $J(\rho_0)$ for $\sigma \in [0, 1]$ is obtained. Thus, $a(\sigma, \rho_0, \partial_\sigma \rho_0)$ is well-defined and positive for $\sigma \in [0, 1]$. Consequently, the operator $\mathcal{A}$ is well-defined and non-trivial.
	
	Now, set $m_1 = 1$ and $m_2 = 3$. It is obvious that the boundary operator $\mathcal{B}_1$ is non-trivial as well, as its coefficient is $i b_1 = i$.
	We proceed with the non-triviality of $\mathcal{B}_2$: Like previously mentioned, we have $J(\rho_0) > 0$ for $\sigma \in \{0, 1\}$. Moreover, by Remark \ref{wohldef}.\ref{wohldef0} we obtain $\langle \Psi_q, R \Psi_\sigma \rangle (\sigma, \rho_0) > 0$ for $\sigma \in \{0, 1\}$. Thus, the coefficient $i^3 b_{2}(\sigma, \rho_0, \partial_\sigma \rho_0)$ is well-defined and non-zero for $\sigma \in \{0, 1\}$, and $\mathcal{B}_2$ is non-trivial.\\
	
	It remains to check the structural assumptions on the operators $(SD)$ and $(SB)$ and the ellipticity conditions $(E)$ and $(LS_{stat})$:
	For $(SD)$ we show that $a \in BUC(\bar{J} \times \bar{I})$. As the initial datum $\rho_0$ and accordingly the coefficient of $\mathcal{A}$, given by $a (\sigma, \rho_0, \partial_\sigma \rho_0)$, does not depend on $t$, the time regularity is trivial. Remark \ref{initial}.\ref{initial1} and Remark \ref{wohldef}.\ref{wohldef0} give a lower and upper bound on $J(\rho_0)$ for $\sigma \in [0 ,1]$, respectively. Moreover, we observe that for $0< a < b < \infty$, the mapping $[a, b] \ni x \mapsto \left(\nicefrac{1}{x}\right)^4$ is continuous. Additionally, as a sum, product and concatenation of continuous functions, the mapping $[0, 1] \ni \sigma \mapsto J(\rho_0)$ is continuous. Thus, 
	$[0, 1] \ni a(\sigma, \rho_0, \partial_\sigma \rho_0)$ is continuous on a compact interval, hence uniformly continuous and bounded. \\ \\
	Next, we want to verify condition $(SB)$: For $j = 1$ we show that the "or"-case is fulfilled, i.e.\ $ib_{1} \in \mathbb{F}_{1, \mu}$. The regularity is trivial and we observe for $\mu \in (\nicefrac78, 1]$
	\begin{align*}
	\omega_1 = \frac{5}{8} = 1 - \frac{7}{8} + \frac{1}{2} > \ 1 - \mu + \frac{1}{2}.
	\end{align*}
	
	For $j = 2$ we prove that the assumptions of the "either"-case hold true, i.e.\ $i^3 b_{2} \in  C^{\tau_2, 4 \tau_2}(\bar{J} \times \partial I)$ with some $\tau_2 > \omega_2 = \nicefrac{1}{8}$. We fix $\tau_2$ such that $1 > 4 \tau_2 > 4 \omega_2$ is fulfilled. The time-regularity is trivial again. As $\partial I$ consists only of two points, the spatial regularity follows directly as well. \\
	
	For verifying the ellipticity condition $(E)$, we prove that for all $t \in \bar{J}$, $x \in \bar{I}$ and $|\xi| = 1$, it holds for the spectrum $\Sigma(\mathcal{A}(t, \sigma, \xi)) \subset \mathbb{C}_+ := \{\Re{z} > 0\}$. The symbol of $\mathcal{A}$ is given by
	\begin{align*}
	\mathcal{A}(t, \sigma, \xi) = i^4 \frac{1}{(J(\rho_0))^4} \xi^4 = \frac{1}{(J(\rho_0))^4} > 0.
	\end{align*}
	Thus, it fulfills the condition by Remark \ref{initial}.\ref{initial1}. \\
	
	Lastly, we check the Lopatinskii-Shapiro-condition $(LS_{stat})$: for each fixed $t \in \bar{J}$ and $\sigma \in \{0, 1\}$, for each $\lambda \in \overline{\mathbb{C}_+} \backslash \{ 0 \}$ and each $h \in \mathbb{C}^{2}$ the ordinary boundary value problem
	\begin{align}
	\lambda v(y) + a(\sigma) \partial_y^4 v(y) &= 0 && y > 0, \label{ode1}\\
	\tilde{b}_2(\sigma) \partial_y^3 v(y) |_{y = 0} &= h_2, && \label{ode2}\\
	\tilde{b}_1(\sigma) \partial_y v(y)  |_{y = 0} &= h_1, && \label{ode3}\\
	\lim_{y \rightarrow \infty} v(y) &= 0, \label{ode4}
	\end{align}
	with 
	\begin{align*}
	\begin{rcases}
	\tilde{b}_2(\sigma) &:= \frac{\pm 1}{(J(\rho_0(\sigma)))^3} \langle \Psi_q, R \Psi_\sigma \rangle (\sigma, \rho_0(\sigma)) \\
	\tilde{b}_1(\sigma) &:= \pm 1,
	\end{rcases} && \textrm{ for } \sigma = 0, 1,
	\end{align*}
	has a unique solution $v \in C_0([0, \infty); \mathbb{C})$, where
	\begin{align*}
	C_0 \left([0, \infty \right); E) := \left\{ f: [0, \infty) \rightarrow E \textrm{ is continuous with } \lim_{t \rightarrow \infty} f(t) = 0 \right\}.
	\end{align*}
	Note that we already proved before, that $\tilde{b}_2$ and $\tilde{b}_1$ do not vanish due to our assumptions on $\rho_0$. Solving the corresponding characteristic equation to \eqref{ode1}, we obtain that a solution is given by
	\begin{align*}
	v(y) =  c_{0} e^{\mu_0 y} + c_{1} e^{\mu_1 y} + c_{2} e^{\mu_2 y} + c_{3} e^{\mu_3 y},
	\end{align*}
	where $c_0, \dots, c_3 \in \mathbb{C}$ and
	\begin{align*}
	\mu_0 &= \sqrt[4]{r} (\bar{a} + i\bar{b}), \hspace{1 cm} &\mu_1 &= \sqrt[4]{r} (-\bar{b} + i\bar{a}), \\
	\mu_2 &= \sqrt[4]{r} (-\bar{a} + i(-\bar{b})), \hspace{1 cm} &\mu_3 &= \sqrt[4]{r} (\bar{b} + i(-\bar{a})).
	\end{align*}
	Here the numbers $r, \bar{a}$, and $\bar{b}$ are determined by the identity
	\begin{align*}
	\frac{- \lambda}{a(\sigma)} = r [\cos \theta + i \sin \theta] && \textrm{ for } \theta \in \left[\frac{\pi}{2}, \frac{3 \pi}{2} \right],
	\end{align*}
	where $\lambda \in \overline{\mathbb{C}_+} \backslash \{ 0 \}$ and $\sigma = 0, 1$. More precisely, it holds
	\begin{align*}
	r &:= \frac{|\lambda|}{a(\sigma)} > 0, && \bar{a} := \cos(\nicefrac{\theta}{4}) > 0, && \bar{b} := \sin(\nicefrac{\theta}{4}) > 0,
	\end{align*}
	for $\nicefrac{\theta}{4} \in [\nicefrac{\pi}{8}, \nicefrac{3 \pi}{8} ]$.
	Due to condition \eqref{ode4} we see immediately that $c_{0} = c_{3} = 0$.
	Thus, the first and third derivative of the solution is given by
	\begin{align*}
	v'(0) =  c_{1} \mu_1 + c_{2} \mu_2  && \textrm{ and } && v'''(0) = c_{1} \mu_1^3 + c_{2} \mu_2^3.
	\end{align*}
	Combining this with \eqref{ode2} and \eqref{ode3}, we have to show that for each $(h_2, h_1) \in \mathbb{C}^{2}$
	\begin{align*}
	M \begin{pmatrix}
	c_{1} \\
	c_{2}
	\end{pmatrix}
	= 
	\begin{pmatrix}
	\nicefrac{h_{2}}{\tilde{b}_{2}} \\
	\nicefrac{h_{1}}{\tilde{b}_{1}}
	\end{pmatrix}, && \textrm{ with } M:= \begin{pmatrix}
	\mu_1^3 & \mu_2^3 \\
	\mu_1 & \mu_2
	\end{pmatrix},
	\end{align*}
	has a unique solution $(c_{1}, c_{2})$. We observe that $\det M = \mu_1 \mu_2 (\mu_1^2 - \mu_2^2) \neq 0$, as $\mu_1 \neq 0 \neq \mu_2$ and 
	\begin{align*}
	\mu_1^2 - \mu_2^2 = \sqrt[2]{r} \left[\left(2 \bar{b}^2 - 2 \bar{a}^2 \right) + i \left(-4 \bar{a}\bar{b}\right) \right] \neq 0, && \textrm{ as } \bar{a}, \bar{b} > 0.
	\end{align*}
	This proves condition $(LS)$.\\ \\
	Finally, we observe 
	for $\mu \in \left(\frac78, 1 \right]$
	\begin{align}
	\omega_1 = \frac{5}{8} = 1 - \frac{7}{8} + \frac{1}{2} > 1 - \mu + \frac{1}{2}, &&
	\omega_2 = \frac{1}{8} < 1 - 1 + \frac{1}{2} \leq 1 - \mu + \frac{1}{2}, \label{dataspace}
	\end{align}
	and therefore, $1 - \mu + \nicefrac{1}{p} \neq \omega_1, \omega_2$. 
\end{proof}

\begin{rem} \label{glm}
	By \eqref{dataspace} it follows directly by the addendum of Theorem 2.1 in \protect{\cite{meyries}} that
	the data space is given by
	\begin{align*}
	\mathcal{D} = \left\{ \left(F, \tilde{G}, \rho_0 \right) \in \mathbb{E}_{0, \mu} \times \tilde{\mathbb{F}}_\mu \times X_{\mu} : \mathcal{B}_1 (0, \cdot, D) \rho_0 = G_1 (0, \cdot) \textrm{ on } \partial I \right\}.
	\end{align*}
	The initial datum $\rho_0$ does not have to fulfill further compatibility conditions than \eqref{comp}.
	Moreover, we observe that the compatibility condition is fulfilled for right-hand side zero, i.e.\ $G_1 (0, \cdot) = 0$, hence $\mathcal{D} = \mathcal{D}_0$ in Theorem 2.1 in \protect{\cite{meyries}}. Consequently, we immediately get a uniform estimate for the solution operator for all $T \in (0, T_0]$, $T_0$ given.
\end{rem}

In order to obtain a solution to the quasilinear problem \eqref{prob} a contraction mapping argument is applied. We consider the linear operator $\mathcal{L}:  \mathbb{E}_{\mu, T} \rightarrow \mathbb{E}_{0, \mu} \times \tilde{\mathbb{F}}_{\mu} \times X_{\mu} $, which is given by the left-hand side of \eqref{linprob}. Moreover, we set
\begin{align}
\mathcal{F}({\rho}) := \begin{pmatrix}
F(t, \sigma) \\
G_1(t, \sigma) \\
G_2(t, \sigma)
\end{pmatrix} =
\begin{pmatrix}
- \left(\frac{1}{(J({\rho}))^4} - \frac{1}{(J(\rho_0))^4}\right) \partial_\sigma^4 {\rho} + f({\rho}, \partial_\sigma {\rho}, \partial_\sigma^2 {\rho}, \partial_\sigma^3 {\rho}) \\
0 \\
- \left(b_2({\rho}, \partial_\sigma {\rho}) - b_2(\rho_0, \partial_\sigma \rho_0)\right) \partial_\sigma^3 {\rho} - g_2({\rho}, \partial_\sigma {\rho}, \partial_\sigma^2 {\rho})
\end{pmatrix},
\label{nonlin}
\end{align}
which corresponds to the right-hand side of \eqref{linprob}. The equation \eqref{linprob} is represented by
\begin{align}
\mathcal{L}(\rho) = (\mathcal{F}({{\rho}}), \rho_0) && \textrm{ for } {\rho} \in \mathcal{B}_{K, T} \textrm{ with } 0 < T < \tilde{T}.
\label{equ}
\end{align}
In order to receive a fixed point equation, Theorem \ref{linex} is used for inverting the linear operator $\mathcal{L}$. Thus, it remains to show that $(F, \tilde{G}) \in \mathbb{E}_{0, \mu} \times \tilde{\mathbb{F}}_{\mu}$.
\begin{lem} \label{lemab}
	Let the assumptions of Theorem \ref{local} hold true and let $K$ and $\tilde{T}$ be given by Definition \ref{ass}. Then it holds 
	\begin{align*}
	\mathcal{F}({\rho}) \in \mathbb{E}_{0, \mu} \times \tilde{\mathbb{F}}_{\mu} &&\textrm{ for } {\rho} \in \mathcal{B}_{K, T} \textrm{ with } 0 < T < \tilde{T}.
	\end{align*}
\end{lem}

\begin{rem}
	Even though it is not important for the proof of this claim, we will keep track of the dependencies of the constants, especially if they depend on $T$. This will enable us to use the derived estimates in the proof of the contraction property of the operator $\mathcal{F}$.
\end{rem}

\begin{proof}[Proof of Lemma \ref{lemab}]
	Firstly, we want to take care of the first component of $\mathcal{F}({\rho})$, cf.\ \eqref{nonlin}: More precisely, we want to show that $F \in \mathbb{E}_{0, \mu}$ for $\rho \in \mathcal{B}_{K, T}$ with $0 < T < \tilde{T}$. To this end, the following claim is proven.
	\begin{claim} \label{lip}
		Let $K$ and $\tilde{T}$ be given by Definition \ref{ass}. For ${\rho} \in \mathcal{B}_{K, T}$ with $0 < T < \tilde{T}$, there exists a constant $C(\alpha, \Phi^*, \eta, K) > 0$, such that
		\begin{align*}
		\left\|\left( \frac{1}{(J(\rho))^4} - \frac{1}{(J(\rho_0))^4} \right) \right\|_{C(\bar{J}, C(\bar{I}))} \leq C(\alpha, \Phi^*, \eta, K) \|{J(\rho)} - J(\rho_0) \|_{C(\bar{J}, C(\bar{I}))},
		\end{align*}
	\end{claim}
	
	\begin{proof}[Proof of the claim:]
		By \eqref{klein}, there exists a  $C(\alpha, \Phi^*, \eta, K)$, such that
		\begin{align}
		\left\| \frac{1}{J(\rho)} \right\|_{C(\bar{J}, C(\bar{I}))} \leq C(\alpha, \Phi^*, \eta, K) && \textrm{ for } {\rho} \in \mathcal{B}_{K, T} \textrm{ with } 0 < T < \tilde{T}.
		\label{1j}
		\end{align}
		Using Remark \ref{wohldef}.\ref{gross}, we have a $\bar{C}(\alpha, \Phi^*, \eta, K) > 0$ such that
		\begin{align}
		\|{J(\rho)} \|_{C(\bar{J}, C(\bar{I}))} \leq \bar{C}(\alpha, \Phi^*, \eta, K) && \textrm{ for } {\rho} \in \mathcal{B}_{K, T}  \textrm{ with } 0 < T < \tilde{T}.
		\label{j}
		\end{align}
		Now, the Lipschitz continuity of $x \mapsto \nicefrac{1}{x^4}$ on the interval $[C^{-1}, \bar{C}]$ is exploited: We denote the corresponding Lipschitz constant by $C(\alpha, \Phi^*, \eta, K)$ and the claim follows directly.
	\end{proof}
	
	For the first summand of $F$, we have by Claim \ref{lip} and Remark \ref{wohldef}.\ref{gross} and \ref{wohldef}.\ref{wohldef0},
	\begin{align*}
	\left\|- \left(\frac{1}{(J(\rho))^4} - \frac{1}{(J(\rho_0))^4} \right) \partial_\sigma^4 \rho \right\|_{L_{2, \mu} (J; L_{2} (I))} &\leq L \|{J(\rho)} - J(\rho_0) \|_{C(\bar{J}, C(\bar{I}))} \|\rho\|_{L_{2, \mu} (J; W^4_{2} (I))} \\
	&\leq C(\alpha, \Phi^*, \eta, K), 
	\end{align*}
	for all $\rho \in \mathcal{B}_{K, T}$ with  $0 < T < \tilde{T}$. \\
	
	We proceed with the estimate of $\|f\|_{\mathbb{E}_{0, \mu}}$ by $\|\rho\|_{ \mathbb{E}_{\mu, T}}$, cf.\ \eqref{darf} for the representation of $f$.
	First, we take care of the prefactors $S(\sigma, \rho, \partial_{\sigma} \rho)$: By considering the structure their structure, cf.\ \eqref{s} and \eqref{tildes}., we want to prove that they are bounded in $L_\infty(\bar{J} \times \bar{I})$ by a constant depending on $\alpha, \Phi^*, \eta$, and $K$. 
	By Remark \ref{wohldef}.\ref{wd1}, we obtain a bound on $\nicefrac{1}{\langle \Psi_q , R \Psi_\sigma \rangle(\sigma, \rho)}$. For the factor $(J(\rho))^n$, $n \in \mathbb{Z}$ we can directly use the bounds established in \eqref{1j} and \eqref{j}, respectively. Moreover, we control $\rho$ and $\partial_{\sigma} \rho$ by Lemma \ref{regu1}. Finally, we take care of the scalar products: Note that there are at most four derivatives on $\Psi$. Combining the $C^4$-bound on $[\sigma \mapsto \Psi(\sigma, q)]$ established in Remark \ref{smooth}.\ref{linf} with the previously discussed bound on $\rho$, we obtain 
	\begin{align*}
	\left\| \left[(t, \sigma) \mapsto \Psi^{\gamma}_{(\sigma, q)} (\sigma, \rho(t, \sigma)) \right] \right\|_{L_{\infty} (\bar{J} \times \bar{I})} \leq C(\alpha, \Phi^*, \eta, K),
	\end{align*}
	for $|\gamma| \leq 4$. In summary, a suitable bound is given by
	\begin{align*}
	\|S(\sigma, \rho, \partial_{\sigma} \rho)\|_{L_{\infty} (\bar{J} \times \bar{I})} \leq C(\alpha, \Phi^*, \eta, K).
	\end{align*}
	
	Next, we find $L_{2, \mu} (J; L_{2} (I))$-bounds for the summands of $f({\rho}, \partial_\sigma {\rho}, \partial_\sigma^2 {\rho}, \partial_\sigma^3 {\rho})$, where the prefactors can be neglected due to the previously established bound. In the following, we will use $\rho_1, \rho_2 \in \mathcal{B}_{K, T}$ to keep the calculations general. We have 
	\begin{align*}
	\left\|\partial^2_\sigma \rho_1 \partial_\sigma^3 \rho_2 \right\|_{L_{2, \mu} (J; L_{2} (I))} &\leq \left\| \left\|\partial^2_\sigma \rho_1 (t)\right\|_{L_{q_1} (I)} \left\|\partial_\sigma^3 \rho_2 (t)\right\|_{L_{q_2} (I)} \right\|_{L_{2, \mu} (J)} \nonumber \\
	&\leq \left\| \|\rho_1(t) \|_{W^2_{q_1} (I)} \|\rho_2(t)\|_{W^3_{q_2} (I)} \right\|_{L_{2 , \mu} (J)},
	\end{align*}
	where the estimates follow for $q_1, q_2 \in [2, \infty]$ fulfilling $\nicefrac{1}{q_1} + \nicefrac{1}{q_2} = \nicefrac{1}{2}$ by Hölder's inequality. 
	By expanding the integrand, we obtain for $\tilde{\mu}_i \in [\mu, 1]$, $i = 1, 2$, 
	\begin{align}
	\left\|\partial^2_\sigma \rho_1 \partial_\sigma^3 \rho_2 \right\|_{L_{2, \mu} (J; L_{2} (I))} &\leq \left\| t^{1-\tilde{\mu}_1}\|\rho_1(t)\|_{W^2_{q_1} (I)} t^{1-\tilde{\mu}_2} \|\rho_2(t)\|_{W^3_{q_2} (I)} t^{1 - \mu - (1-\tilde{\mu}_1 + 1-\tilde{\mu}_2)} \right\|_{L_{2} (J)} \nonumber\\
	&\leq C_1 \left\| t^{1-\tilde{\mu}_1}\|\rho_1(t)\|_{W^2_{q_1} (I)} t^{1-\tilde{\mu}_2} \|\rho_2(t)\|_{W^3_{q_2} (I)} \right\|_{L_{2} (J)} \nonumber \\
	&\leq C_1 C_2 \|\rho_1 \|_{L_{l_1 , \tilde{\mu}_1} (J; W^3_{q_1} (I))} \|\rho_2 \|_{L_{l_2, \tilde{\mu}_2} (J; W^3_{q_2} (I))}, \label{contra}
	\end{align}
	for $\nicefrac{1}{l_1} + \nicefrac{1}{l_2} \leq \nicefrac{1}{2}$. Here 
	\begin{align*}
	C_1 &= \begin{cases}
	C(T) \rightarrow 0 \textrm{ as } T \rightarrow 0 \hspace{3cm} & \textrm{ if } 1 - \mu - (1-\tilde{\mu}_1 + 1-\tilde{\mu}_2) > 0, \\
	1 & \textrm{ if } 1 - \mu - (1-\tilde{\mu}_1 + 1-\tilde{\mu}_2) = 0,
	\end{cases} \\
	C_2 &= \begin{cases}
	C(T) \rightarrow 0 \textrm{ as } T \rightarrow 0 \hspace{3cm} & \textrm{ if } \frac{1}{l_1} + \frac{1}{l_2} < \frac{1}{2}, \\
	1 & \textrm{ if } \frac{1}{l_1} + \frac{1}{l_2} = \frac{1}{2}.
	\end{cases}
	\end{align*}
	In order to control $\|\partial^2_\sigma \rho_1 \partial_\sigma^3 \rho_2 \|_{L_{2, \mu} (J; L_{2} (I))}$ in terms of $\|\rho_i\|_{\mathbb{E}_{\mu, T}}$, Proposition \ref{embed} is used to estimate the right-hand side of \eqref{contra}, where $\nicefrac{1}{q_1} + \nicefrac{1}{q_2} = \nicefrac{1}{2}$, $k_1 = 2, k_2 = 3$, and $l_i$ and the time-weights $\tilde{\mu}_i$ have to be chosen carefully: By the assumptions of the Proposition \ref{embed}, we obtain by $\nicefrac{\theta_i}{2} = \nicefrac{1}{l_i}$, $i =1, 2$,
	\begin{align*}
	\theta_1 + \theta_2 < 1  && \Leftrightarrow && \frac{1}{l_1} + \frac{1}{l_2} < \frac{1}{2} && \textrm{ and } && \theta_1 + \theta_2 = 1  && \Leftrightarrow && \frac{1}{l_1} + \frac{1}{l_2} = \frac{1}{2}.
	\end{align*}
	Moreover, we have by $\tilde{\mu}_i = \mu + (1 - \theta_i)(1 - \mu) \in [\mu, 1]$ the equality
	\begin{align}
	1 - \mu - (1-\tilde{\mu}_1 + 1-\tilde{\mu}_2)
	= (1 - (\theta_1 + \theta_2))(1 - \mu),
	\label{hilfe1}
	\end{align}
	thus, for $\mu \in (\nicefrac{7}{8}, 1)$
	\begin{align*}
	\theta_1 + \theta_2 &< 1 \; \Leftrightarrow \; 1 - \mu - (1-\tilde{\mu}_1 + 1-\tilde{\mu}_2) > 0, \\
	\theta_1 + \theta_2 &= 1 \; \Leftrightarrow \; 1 - \mu - (1-\tilde{\mu}_1 + 1-\tilde{\mu}_2) = 0,
	\end{align*}
	and $1 - \mu - (1-\tilde{\mu}_1 + 1-\tilde{\mu}_2) = 0$ for $\mu = 1$.
	Consequently, to produce a constant $C(T) \rightarrow 0$ for $T \rightarrow 0$, it remains to prove that for $\nicefrac{1}{q_1} + \nicefrac{1}{q_2} = \nicefrac{1}{2}$, $k_1 = 2, k_2 = 3$ and $\mu \in (\nicefrac{7}{8}, 1]$, it is possible to fulfill $\theta_1 + \theta_2 < 1$.
	Using $s_i := k_i + \nicefrac{1}{2} - \nicefrac{1}{q_i} = 4\left(\mu - \nicefrac{1}{2}\right) (1 - \theta_i) + 4 \theta_i$, we directly calculate
	\begin{align*}
	\theta_1 + \theta_2 
	= 2 - \frac{5}{8 \left(1 - \mu + \frac{1}{2}\right)} < 2 - \frac{5}{8 \left(1 - \frac{7}{8} + \frac{1}{2}\right)} = 1. 
	\end{align*}
	Hence, applying Proposition \ref{embed} to \eqref{contra}, it follows
	\begin{align}
	\left\|\partial^2_\sigma \rho_1 \partial_\sigma^3 \rho_2 \right\|_{L_{2, \mu} (J; L_{2} (I))}
	&\leq C(T) \prod_{i=1}^{2} \left(\|\rho_i\|_{L_{\infty}(J, W_{2}^{4(\mu - \nicefrac{1}{2})}(I))} + \| \rho_i \|_{L_{2, \mu}(J, W^4_2 (I))} \right),
	\label{contr}
	\end{align}
	where $C(T) \rightarrow 0$ as  $T \rightarrow 0$, as the constant of the embedding in Proposition \ref{embed} does not depend on $T$. Setting $\rho_i = \rho$, $i = 1, 2$, in \eqref{contr} and using the embedding of the solution space into the temporal trace space, see Lemma \ref{unabneu}.\ref{unab1}, we deduce
	\begin{align*}
	\left\|S(\sigma, \rho, \partial_\sigma \rho) \partial^2_\sigma \rho \partial_\sigma^3 \rho \right\|_{L_{2, \mu} (J; L_{2} (I))} \leq C(\alpha, \Phi^*, \eta, K, T) && \textrm{ for } \rho \in \mathcal{B}_{K, T} \textrm{ with } 0 < T < \tilde{T}.
	\end{align*}
	Note that the constant on the right-hand side does in general no longer fulfill $C(T) \rightarrow 0$ for $T \rightarrow 0$.
	
	By inspection of the remaining summands, we obtain that all terms besides $S(\partial^2_\sigma \rho)^3$ can be treated as in \eqref{contr}, where an additional $T^s$ with $s>0$ is possibly produced. Therefore, it just remains to show that $(\partial^2_\sigma \rho)^3 \in L_{2, \mu} (J; L_{2} (I))$ for $\rho \in \mathcal{B}_{K, T}$ with $0 < T < \tilde{T}$. To keep the calculations general for employing them later, we use again $\rho_i$, $i = 1, 2, 3$. By Hölder's inequality and direct estimates, we deduce
	\begin{align*}
	\left\|\prod_{i=1}^{3}\partial^2_\sigma \rho_i \right\|_{L_{2, \mu} (J; L_{2} (I))} 
	&\leq \left\| t^{1 - \mu} \prod_{i=1}^{3} \left\|\partial^2_\sigma \rho_i (t) \right\|_{ L_{6} (I)} \right\|_{L_{2} (J)} \leq \left\| t^{1 - \mu} \prod_{i=1}^{3} \| \rho_i (t) \|_{ W^2_{6} (I)} \right\|_{L_{2} (J)}.
	\end{align*}
	Again, we want to find bounds for the right-hand side by Proposition \ref{embed}, where $k = k_i = 2$ and $q = q_i = 6$, $i = 1, 2, 3$. A direct calculation shows that $\theta = \theta_i $ is given by
	\begin{align}
	\theta = \frac{k + \frac{1}{2} - \frac{1}{q} - 4 \left(\mu - \frac{1}{2}\right)}{4 \left(\frac{3}{2} - \mu \right)} = \frac{4 + \frac{1}{3} - 4 \mu}{4 \left(\frac{3}{2} - \mu \right)} \in \left[\frac{1}{6}, \frac{1}{3}\right) && \textrm{ for } \mu \in \left(\frac{7}{8}, 1 \right].
	\label{theta}
	\end{align}
	Moreover, by $\nicefrac{\theta}{2} = \nicefrac{1}{l}$ and analogously to \eqref{hilfe1}, we have for $\mu \in (\nicefrac{7}{8}, 1)$
	\begin{align*}
	\frac{3}{l} < \frac{1}{2} && \Leftrightarrow && 3 \theta < 1 && \Leftrightarrow && 1 - \mu - 3(1 - \tilde{\mu}) > 0,
	\end{align*}
	and $ 1 - \mu - 3(1 - \tilde{\mu}) = 0$ for $\mu = 1$, where $\tilde{\mu} = \tilde{\mu}_i = \mu + (1 - \theta_i)(1 - \mu) \in [\mu, 1]$.
	Using \eqref{theta}, we deduce $3 \theta < 1$ for $\mu \in (\nicefrac{7}{8}, 1]$ and consequently we obtain for $l > 6$
	\begin{align*}
	\left\|\prod_{i=1}^{3}\partial^2_\sigma \rho_i \right\|_{L_{2, \mu} (J; L_{2} (I))} 
	&\leq \left\| t^{1 - \mu-3(1 - \tilde{\mu})} \left(t^{1 - \tilde{\mu}} \right)^3 \prod_{i=1}^{3} \| \rho_i (t) \|_{ W^2_{6} (I)} \right\|_{L_{2} (J)}  \nonumber \\
	&\leq C(T) \prod_{i=1}^{3} \left\|t^{1 - \tilde{\mu}}\| \rho_i (t) \|_{ W^2_{6} (I)} \right\|_{L_{l} (J)} = C(T) \prod_{i=1}^{3} \|\rho_i \|_{L_{l, \tilde{\mu}} (J;  W^2_{6} (I))}, 
	\end{align*}
	where $C(T) \rightarrow 0$ as  $T \rightarrow 0$.
	By Proposition \ref{embed}, it follows
	\begin{align}
	\left\|\prod_{i=1}^{3}\partial^2_\sigma \rho_i \right\|_{L_{2, \mu} (J; L_{2} (I))} \leq C(T) \prod_{i=1}^{3} \left(\|\rho_i\|_{L_{\infty}(J, W_{2}^{4(\mu - \nicefrac{1}{2})}(I))} \| + \| \rho_i \|_{L_{2, \mu}(J, W^4_2 (I))} \right),
	\label{contr1}
	\end{align}
	where the constant still fulfills $C(T) \rightarrow 0$ for $T \rightarrow 0$, since the operator norm of the embedding in Proposition \ref{embed} does not depend on $T$. Setting $\rho_i = \rho$, $i = 1, 2, 3$, and using again the embedding in Lemma \ref{unabneu}.\ref{unab1}, we obtain
	\begin{align*}
	\left\|S(\sigma, \rho, \partial_\sigma \rho) (\partial^2_\sigma \rho)^3 \right\|_{L_{2, \mu} (J; L_{2} (I))} \leq C(\alpha, \Phi^*, \eta, K, T) && \textrm{ for } \rho \in \mathcal{B}_{K, T} \textrm{ with } 0 < T < \tilde{T}.
	\end{align*}
	\\
	Furthermore, we have to verify that $(G_1, G_2) \in \tilde{\mathbb{F}}_{\mu}$ for  $\rho \in \mathcal{B}_{K, T}$:
	We observe that it suffices to verify $G_1 \in W^{\nicefrac{5}{8}}_{2, \mu} (J)$ and $G_2 \in W^{\nicefrac{1}{8}}_{2, \mu} (J)$ for $\sigma \in \{0, 1\}$, as the boundary $\partial I$ only consists of two points. Note that the claim for $G_1 \equiv 0$ is trivially fulfilled. 
	
	First, we consider the regularity of $\rho$ and its derivatives at the boundary $\partial I$: Combining Lemma \ref{unabneu}.\ref{unab234} and \ref{unabneu}.\ref{unab5678}, we obtain $\partial^k_\sigma \rho \in W^{\nicefrac{7 - 2k}{8}}_{2, \mu} (J; L_{2} (\partial I) ) \cap L_{2, \mu} (J; W^{\nicefrac{7 - 2k}{2}}_{2} (\partial I) )$ for $k = 0, \dots, 3$ as $\rho \in  \mathbb{E}_{\mu, T}$,
thus,  
	\begin{align}
	\partial^k_\sigma \rho(\cdot, \sigma) &\in W^{\nicefrac{(7-2k)}{8}}_{2, \mu} (J)  
	&&\textrm{ for } k = 0, \dots, 3 \textrm{ and } \sigma \in \{0, 1\}.
	\label{rhoreg}
	\end{align} 
	This enables us to prove $G_2 \in W^{\nicefrac{1}{8}}_{2, \mu} (J)$, see \eqref{nonlin}, for $\sigma \in \{0, 1\}$.
	see \eqref{nonlin}. 
	The summand $b_2(\sigma, \rho_0, \partial_\sigma \rho_0) \partial_\sigma^3 {\rho}$ is clearly an element of $W^{\nicefrac{1}{8}}_{2, \mu} (J)$, since the coefficient does not depend on $t$. In order to prove the regularity for the other summands, Lemma \ref{ban-alg}.\ref{lem1} and Remark \ref{banrem} are used: The previously mentioned results state that the product of an element of $W^{\nicefrac{5}{8}}_{2, \mu} (J)$ with one of $W^{\nicefrac{1}{8}}_{2, \mu} (J)$ is again in $W^{\nicefrac{1}{8}}_{2, \mu} (J)$. Therefore, it is helpful to prove
	\begin{align}
	\begin{rcases}
	b_2(\sigma, \rho, \partial_\sigma \rho) &\in W^{\nicefrac{5}{8}}_{2, \mu} (J), \\
	T(\sigma, \rho, \partial_\sigma \rho) &\in W^{\nicefrac{5}{8}}_{2, \mu} (J)
	\end{rcases}
	&& \textrm{ for } \sigma \in \{0, 1\}, 
	\label{beh}
	\end{align}
	see \eqref{coeffs} for the definition of $b_2$ and \eqref{darg} for the representation of $g_2$ with the coefficients $T$, cf.\ \eqref{t}. In order to achieve this, we show the following claims. 
	
	\begin{claim} \label{lem3a}
		For $\rho \in \mathcal{B}_{K, T}$ with $0 < T < \tilde{T}$, it holds $J(\rho) \in W^{\nicefrac{5}{8}}_{2, \mu} (J)$ and $\tfrac{1}{J(\rho)} \in W^{\nicefrac{5}{8}}_{2, \mu} (J)$ for $\sigma \in \{0, 1\}$ with the estimates
		\begin{align*}
		\begin{rcases}
		\left\| J(\rho) \right\|_{W^{\nicefrac{5}{8}}_{2, \mu} (J)} \leq C(\alpha, \Phi^*, \eta, K, \tilde{T}), \\
		\left\|\tfrac{1}{J(\rho)} \right\|_{W^{\nicefrac{5}{8}}_{2, \mu} (J)} \leq C(\alpha, \Phi^*, \eta, K, \tilde{T}) 
		\end{rcases}
		&& \textrm{ for } \rho \in \mathcal{B}_{K, T} \textrm{ with } 0 < T < \tilde{T}  \textrm{ and } \sigma \in \{0, 1\}.
		\end{align*}
	\end{claim}
	
	\begin{proof}[Proof of the claim:]
		First, we show $J(\rho) \in W^{\nicefrac{5}{8}}_{2, \mu} (J)$ for $\sigma \in \{0, 1\}$. 
		Recall that $J(\rho)$ is given by
		\begin{align*}
		J(\rho) = |\Phi_\sigma| = |\Psi_\sigma + \Psi_q \partial_{\sigma} \rho| = \sqrt{|\Psi_\sigma|^2 + 2\langle \Psi_\sigma, \Psi_q \rangle \partial_{\sigma} \rho + |\Psi_q|^2(\partial_{\sigma} \rho)^2}.  
		\end{align*}
		We notice that $\Psi_\sigma(\sigma, \rho(\cdot, \sigma)) \in W^{\nicefrac{5}{8}}_{2, \mu} (J)$ for $\sigma \in \{0, 1\}$, since its summands are either independent of $t$ or are given by $\rho(t, \sigma) v$ for
		$v \in \mathbb{R}^2$, which does not depend on $t$. 
		Due to the regularity of $\rho$ in \eqref{rhoreg}, the claim follows directly by the definition of the norm. Besides, $\Psi_q(\sigma) \partial_\sigma \rho(\cdot, \sigma) \in W^{\nicefrac{5}{8}}_{2, \mu} (J)$ for $\sigma \in \{0, 1\}$ as well by \eqref{rhoreg}, since $\Psi_q$ does not depend on $t$. Moreover, we have 
		\begin{align*}
		|\langle \Psi_\sigma, \Psi_q \rangle \partial_\sigma \rho| \leq |\Psi_\sigma| |\Psi_q| \partial_\sigma \rho| \leq |\Psi_\sigma|^2 + (|\Psi_q| \partial_\sigma \rho)^2
		\end{align*}
		by Cauchy-Schwartz-inequality and Young's inequality. Using this, we obtain directly
		\begin{align*}
		\|J(\rho)\|^2_{L_{2, \mu} (J)} 
		&\leq 2 \int_{J} t^{2(1 - \mu)} \left(|\Psi_\sigma|^2 + |\Psi_q|^2(\partial_\sigma \rho)^2 \right) \dd t
		\end{align*}
		for $\sigma \in \{0, 1\}$. For the semi-norm, the properties of the scalar product for $\Psi_\sigma(t) = \Psi_\sigma(\sigma, \rho(t, \sigma))$ yield
		\begin{align*}
		[J(\rho)]^2_{W^{\nicefrac{5}{8}}_{2, \mu} (J)} 
		&\leq C \int_{0}^{T} \int_{0}^{t} \tau^{2 (1 - \mu)} \frac{|\Psi_\sigma(t) - \Psi_\sigma(\tau)|^2 + |\Psi_q(\partial_\sigma \rho(t) - \partial_\sigma \rho(\tau))|^2}{|t - \tau|^{1 + 2 \frac{5}{8}}} \dd \tau \dd t, 
		\end{align*}
		where we used $(a+b)^2 \leq 2(a^2 + b^2)$. Note that $\Psi_q$ does not depend on $t$, as it is independent of $\rho$. This shows the first claim and the first estimate. Combining this result with Lemma \ref{regu1}, the second claim and the estimate follows directly by Lemma \ref{ban-alg}.\ref{lem3}.
	\end{proof}
	
	\begin{claim} \label{lem4}
		For $\rho \in \mathcal{B}_{K, T}$ with $0 < T < \tilde{T}$, it holds for $\sigma \in \{0, 1\}$ that
		\begin{align*}
		\left\langle \Psi^{\beta}_{(\sigma, q)} , R \Psi^{\gamma}_{(\sigma, q)} \right\rangle (\sigma, \rho) \in W^{\nicefrac{5}{8}}_{2, \mu} (J)
		\end{align*}
		for $\beta, \gamma \in \mathbb{N}_0^2$, such that $|\beta|, |\gamma| \geq 1$ and $|\beta| + |\gamma| \leq 4$. The $W^{\nicefrac{5}{8}}_{2, \mu} (J)$-norm of the quantity is bounded by a constant $C(\alpha, \Phi^*, \eta, K, \tilde{T})$.
	\end{claim}
	
	\begin{proof}[Proof of the claim] By Lemma \ref{ban-alg}.\ref{lem2} and Remark \ref{banrem}, we know that $W^{\nicefrac{5}{8}}_{2, \mu} (J)$ is a Banach algebra up to a constant in the norm estimate. Thus, it suffices to prove $\Psi^{\beta}_{(\sigma, q)}(\sigma, \rho(\cdot, \sigma)) \in W^{\nicefrac{5}{8}}_{2, \mu} (J)$ for $\sigma \in \{0, 1\}$, where $1 \leq |\beta| \leq 4$. Considering the structure of the mapping $[(\sigma, q) \mapsto \Psi (\sigma, q)]$, we obtain that the term has the form
		\begin{align*}
		\Psi^{\beta}_{(\sigma, q)}(\sigma, \rho(t, \sigma))  = \begin{cases}
		v_1 & \textrm{ if } \beta_2 \geq 1, \\
		v_1 + v_2 \rho(t, \sigma), & \textrm{ else, } 
		\end{cases}
		\end{align*}
		where $v_1, v_2 \in \mathbb{R}^2$ are independent of $t$. The claim follows by the regularity properties of $\rho$, see \eqref{rhoreg} and Lemma \ref{ban-alg}.\ref{lem2} together with Remark \ref{banrem}.
	\end{proof}
	Now, we are ready to prove \eqref{beh}: Combining the results of Claim \ref{lem3a} and Claim \ref{lem4} with the Banach algebra property of $W^{\nicefrac{5}{8}}_{2, \mu} (J)$, see Lemma \ref{ban-alg}.\ref{lem2} and Remark \ref{banrem}, we obtain $b_2(\sigma, \rho, \partial_\sigma \rho) \in W^{\nicefrac{5}{8}}_{2, \mu} (J)$ for $\sigma \in \{0, 1\}$
	with a suitable estimate. We recall that $\partial_\sigma^3 \rho(\cdot, \sigma) \in W^{\nicefrac{1}{8}}_{2, \mu} (J)$ for $\sigma \in \{0, 1\}$ by \eqref{rhoreg}. As the product of an element in $W^{\nicefrac{5}{8}}_{2, \mu} (J)$ with one in $W^{\nicefrac{1}{8}}_{2, \mu} (J)$ is again in $W^{\nicefrac{1}{8}}_{2, \mu} (J)$, see Lemma \ref{ban-alg}.\ref{lem2}, we deduce that the first summand of $G_2$ is an element of $W^{\nicefrac{1}{8}}_{2, \mu} (J)$ for $\sigma \in \{0, 1\}$. \\
	
	It remains to consider $g_2(\sigma, \rho, \partial_\sigma \rho, \partial_\sigma^2 \rho)$,	cf.\ \eqref{darg}. We want to prove that the prefactors $T(\sigma, \rho, \partial_\sigma \rho)$ are in $W^{\nicefrac{5}{8}}_{2, \mu} (J)$ and the terms $(\partial^2_\sigma \rho)^2$ and $\partial^2_\sigma \rho$ are elements of $W^{\nicefrac{1}{8}}_{2, \mu} (J)$ for $\sigma \in \{0, 1\}$. Then, the claim follows by Lemma \ref{ban-alg}.\ref{lem1} and Remark \ref{banrem}, which state that the products of these functions are again in $W^{\nicefrac{1}{8}}_{2, \mu} (J)$ for $\sigma \in \{0, 1\}$ with a corresponding estimate. 
	
	Recalling the structure of $T(\sigma, \rho, \partial_{\sigma} \rho)$, cf.\eqref{t}, the claim follows directly by combining Claim \ref{lem3a}, Claim \ref{lem4}, and the regularity of $\rho$, see \eqref{rhoreg}, Lemma \ref{ban-alg}.\ref{lem2a} and Remark \ref{banrem}. 
	
	Finally, we take care of the factors $(\partial^2_\sigma \rho)^2$ and $\partial^2_\sigma \rho$, respectively: By \eqref{rhoreg}, we obtain $\partial^2_\sigma \rho \in W^{\nicefrac{3}{8}}_{2, \mu} (J)$ for $\sigma \in \{0, 1\}$. 
	Additionally, it follows by Lemma \ref{ban-alg}.\ref{lem2a} that $(\partial^2_\sigma \rho)^2 \in W^{\nicefrac{1}{8}}_{2, \mu} (J)$ for $\sigma \in \{0, 1\}$.
	This completes the proof.
\end{proof}

\begin{rem}
	Combining Theorem \ref{linex} and Lemma \ref{lemab}, we obtain that the linear problem \eqref{linprob} possesses a unique solution $\rho \in  \mathbb{E}_{\mu, T}$, for $\bar{\rho} \in \mathcal{B}_{K, T}$ for $0 < T < \tilde{T}$, if the conditions on the initial datum are fulfilled. This enables us to invert the linear operator in equation \eqref{equ} and consequently we receive the fixed point problem
	\begin{align*}
	\rho = \mathcal{L}^{-1}(\mathcal{F}({\rho}), \rho_0) && \textrm{ for } \rho \in \mathcal{B}_{K, T} \textrm{ with } 0 < T < \tilde{T}, 
	\end{align*}
	which is to be solved by Banach's fixed point theorem.
\end{rem}

\section{The Contraction Mapping} \label{cont}

The next step is to show that the nonlinear operator is contractive. To this end, we state the following lemma:
\begin{lem} \label{contraction}
	Let the assumptions of Theorem \ref{local} hold true and let $K$ and $\tilde{T}$ be given by Definition \ref{ass}. Then
	\begin{align*}
	\mathcal{F}: \mathcal{B}_{K, T} \rightarrow \mathbb{E}_{0, \mu} \times \tilde{\mathbb{F}}_{\mu} &&\textrm{ for } 0 < T < \tilde{T},
	\end{align*}
	given by \eqref{nonlin},
	is Lipschitz continuous with a constant $C_{\mathcal{F}_T} = C(\alpha, \Phi^*, \eta, K, R, \tilde{T}, T)$  with $\|\rho_0\|_{X_{\mu}} \leq R$ satisfying $C_{\mathcal{F}_T} \rightarrow 0$ monotonically as $T \rightarrow 0$. 
\end{lem}

\begin{proof}
	Let $\rho_1$ and $\rho_2$ be in $\mathcal{B}_{K, T}$ for $0 < T < \tilde{T}$.\\
	First, we concentrate on the first line of $\mathcal{F}$. By adding zeros, it holds
	\begin{align*}
	F(\rho_1) - F(\rho_2) 
	&= \underbrace{- \left( \frac{1}{(J(\rho_1))^4} - \frac{1}{(J(\rho_0))^4} \right) \left( \partial_\sigma^4 \rho_1 - \partial_\sigma^4 \rho_2 \right)}_{=: I} \underbrace{-  \left( \frac{1}{(J(\rho_1))^4} - \frac{1}{(J(\rho_2))^4} \right) \partial_\sigma^4 \rho_2}_{=: II} \\ 
	&~~ + \underbrace{\left( f(\rho_1, \partial_\sigma \rho_1, \partial_\sigma^2 \rho_1, \partial_\sigma^3 \rho_1) - f(\rho_2, \partial_\sigma \rho_2, \partial_\sigma^2 \rho_2, \partial_\sigma^3 \rho_2) \right).}_{=: III}
	\end{align*}
	We begin with the first factors of $I$ and $II$, which can be estimated similarly. For $II$, we obtain analogously to Claim \ref{lip}
	\begin{align*}
	\left\|\frac{1}{(J(\rho_1))^4} - \frac{1}{(J(\rho_2))^4} \right\|_{C(\bar{J}; C(\bar{I}))} &\leq L \|J(\rho_1) - J(\rho_2)\|_{C(\bar{J}; C(\bar{I}))},
	\end{align*}
	where $L$ does not depend on $T$. We recall that for $\rho \in \mathcal{B}_{K, T}$, $0 < T < \tilde{T}$, by Lemma \ref{regu1} it follows that $\rho$ and $\partial_{\sigma} \rho$ are bounded in $C^0([0, 1] \times [0, T])$ by $\nicefrac{2K_0}{3}$ and $\nicefrac{2K_1}{3}$, respectively. Thus, we can exploit the Lipschitz continuity of
	\begin{align*}
	[0, 1] \times [-\nicefrac{2K_0}{3}, \nicefrac{2K_0}{3}] \times [-\nicefrac{2K_1}{3}, \nicefrac{2K_1}{3}] \ni (\sigma, \rho, \partial_\sigma \rho) \mapsto J(\rho)
	\end{align*}
	for $\sigma \in \bar{I}$, $\rho \in \mathcal{B}_{K, T}$, which follows by the fact that $J(\rho)$ is continuously differentiable with respect to the variables $(\sigma, \rho, \partial_\sigma \rho)$. The Lipschitz constant depends on $\alpha, \Phi^*, \eta$, and $K$, but not on $T$. We obtain
	\begin{align*}
	\left\|\frac{1}{(J(\rho_1))^4} - \frac{1}{(J(\rho_2))^4} \right\|_{C(\bar{J}; C(\bar{I}))} 
	&\leq C(\alpha, \Phi^*, \eta, K) \|(\rho_1, \partial_\sigma \rho_1) - (\rho_2, \partial_\sigma \rho_2)\|_{C(\bar{J}; C(\bar{I}))} \nonumber \\
	&\leq C(\alpha, \Phi^*, \eta, K) \|\rho_1 - \rho_2\|_{C(\bar{J}; C^1(\bar{I}))} \\ 
	&\leq C(\alpha, \Phi^*, \eta, K) T^{\alpha} \|\rho_1 - \rho_2\|_{C^{\alpha}(\bar{J}; C^1(\bar{I}))},
	\end{align*}
	where $C(\alpha, \Phi^*, \eta, K)$ does not depend on $T$. Here we additionally used that $\rho_{i |t =0} = \rho_0$ for $i = 1, 2$.
	Replacing $\rho_2$ by $\rho_0$, we obtain for summand $I$
	\begin{align*}
	\left\|\frac{1}{(J(\rho_1))^4} - \frac{1}{(J(\rho_0))^4} \right\|_{C(\bar{J}; C(\bar{I}))} 
	\leq C(\alpha, \Phi^*, \eta, K) T^{\alpha} \|\rho_1 - \rho_0\|_{C^{\alpha}(\bar{J}; C^1(\bar{I}))}.
	\end{align*}
	Next, the embedding 
	\begin{align*}
	\mathbb{E}_{\rho, \mu, T} \hookrightarrow C^{{\alpha}} \left(\bar{J}; C^1(\bar{I}) \right)
	\end{align*}
	is employed, where the operator norm does not depend on $T$, if a suitable norm is used,
	cf.\ Lemma \ref{unabneu}.\ref{unab1a}. Therefore, it holds
	\begin{align*}
	\|I\|_{L_{2, \mu} (J; L_{2}(I))} 
	&\leq C(\alpha, \Phi^*, \eta, K) T^{\alpha} \|\rho_1\|_{C^{\alpha}(\bar{J}; C^1(\bar{I}))} \|\rho_1 - \rho_2\|_{ \mathbb{E}_{\mu, T}} \\ 
	&\leq C(\alpha, \Phi^*, \eta, K) T^{\alpha} \tilde{C} (\|\rho_1\|_{\mathbb{E}_{\mu, T}} + \|{\rho_{1}}_{|t=0}\|_{X_{\mu}}) \|\rho_1 - \rho_2\|_{ \mathbb{E}_{\mu, T}} \\ 
	&\leq C(\alpha, \Phi^*, \eta, K, \|\rho_{0}\|_{X_{\mu}}, T) \|\rho_1 - \rho_2\|_{\mathbb{E}_{\mu, T}},
	\end{align*}
	where we used that $\rho_1 \in \mathcal{B}_{K, T}$. Note that $C(\alpha, \Phi^*, \eta, K, \|\rho_{0}\|_{X_{\mu}}, T) \rightarrow 0$ monotonically as $T \rightarrow 0$. For the second summand $II$, it follows analogously
	\begin{align*}
	\|II\|_{L_{2, \mu} (J; L_{2}(I))} 
	&\leq C(\alpha, \Phi^*, \eta, K, T) \|\rho_1 - \rho_2\|_{\mathbb{E}_{\mu, T}}.
	\end{align*}
	In the following, we take care of summand $III$: By adding zeros, we obtain
	\begin{align}
	f(\rho_1, \partial_\sigma \rho_1, & \; \partial_\sigma^2 \rho_1, \partial_\sigma^3 \rho_1) - f(\rho_2, \partial_\sigma \rho_2, \partial_\sigma^2 \rho_2, \partial_\sigma^3 \rho_2) \nonumber \\ 
	&= S(\rho_1) \left(\partial^2_\sigma \rho_1 \partial_\sigma^3 \rho_1 - \partial^2_\sigma \rho_2 \partial_\sigma^3 \rho_2 \right) + \left(S(\rho_1) - S(\rho_2)\right) \partial^2_\sigma \rho_2 \partial_\sigma^3 \rho_2 \nonumber \\
	&~~ + S(\rho_1) \left(\partial_\sigma^3 \rho_1 - \partial_\sigma^3 \rho_2\right) + (S(\rho_1) - S(\rho_2)) \partial_\sigma^3 \rho_2 + S(\rho_1) \left((\partial^2_\sigma \rho_1)^3 - (\partial^2_\sigma \rho_2)^3 \right)  \nonumber \\
	&~~ + (S(\rho_1)- S(\rho_2)) \left(\partial^2_\sigma \rho_2\right)^3 + S(\rho_1) \left((\partial_\sigma^2 \rho_1)^2 - (\partial_\sigma^2 \rho_2)^2\right) \nonumber \\
	&~~ + (S(\rho_1) - S(\rho_2)) (\partial_\sigma^2 \rho_2)^2 + S(\rho_1) \left(\partial_\sigma^2 \rho_1 - \partial_\sigma^2 \rho_2 \right) + (S(\rho_1) - S(\rho_2)) \partial_\sigma^2 \rho_2 \nonumber \\
	&~~ + (S(\rho_1) - S(\rho_2)).
	\label{eg}
	\end{align}
	where we denote by ${S}(\rho_i) = S(\sigma, \rho_i, \partial_{\sigma} \rho_i)$, see \eqref{s} and \eqref{tildes}.
	
	First, we inspect the summands which have a factor $(S(\rho_1) - S(\rho_2))$.  
	We observe that for every $\sigma \in [0, 1]$ the mapping
	\begin{align*}
	[-\nicefrac{2K_0}{3}, \nicefrac{2K_0}{3}] \times [-\nicefrac{2K_1}{3}, \nicefrac{2K_1}{3}] \ni (\rho, \partial_\sigma \rho) \mapsto S(\sigma, \rho, \partial_\sigma \rho)
	\end{align*}
	is Lipschitz-continuous with a Lipschitz constant depending on $\alpha, \Phi^*, \eta$, and $K$. By this, we obtain by the same strategy as for the summands $I$ and $II$
	\begin{align}
	\|S(\rho_1) - S(\rho_2)\|_{C(\bar{J}; C(\bar{I}))} 
	&\leq C(\alpha, \Phi^*, \eta, K, T) \|\rho_1 - \rho_2\|_{C^{\alpha}(\bar{J}; C^1(\bar{I}))},
	\label{coef}
	\end{align}
	where $C(\alpha, \Phi^*, \eta, K, T) \rightarrow 0$ monotonically for $T \rightarrow 0$.
	Furthermore, we have
	\begin{align}
	\|S(\rho_1)\|_{C(\bar{J}; C((\bar{I}))} &\leq \|S(\rho_1) - S(\rho_0)\|_{C(\bar{J}; C((\bar{I}))} + \|S(\rho_0)\|_{C(\bar{J}; C(\bar{I}))} \nonumber \\
	&\leq C(\alpha, \Phi^*, \eta, K) \|\rho_1 - \rho_0\|_{C(\bar{J}; C^1((\bar{I}))} + \|S(\rho_0)\|_{C(\bar{J}; C(\bar{I}))} \nonumber \\
	&\leq C(\alpha, \Phi^*, \eta, K) \left( \|\rho_1\|_{ \mathbb{E}_{\mu, T}} +  \|\rho_0\|_{X_\mu} + \|\rho_0\|_{C^1(\bar{I})} + 1 \right) \nonumber \\
	&\leq C(\alpha, \Phi^*, \eta, K, \|\rho_0\|_{X_\mu}),
	\label{coef1}
	\end{align}
	where $C(\alpha, \Phi^*, \eta, K, \|\rho_0\|_{X_\mu})$ does not depend on $T$.
	By adding a zero, the first summand of \eqref{eg} reads
	\begin{align*}
	S(\rho_1) \left(\partial^2_\sigma \rho_1 \partial_\sigma^3 \rho_1 - \partial^2_\sigma \rho_2 \partial_\sigma^3 \rho_2 \right) 
	= S(\rho_1) \partial^2_\sigma \rho_1 (\partial_\sigma^3 \rho_1 - \partial_\sigma^3 \rho_2) + S(\rho_1) (\partial^2_\sigma \rho_1 - \partial^2_\sigma \rho_2) \partial_\sigma^3 \rho_2
	\end{align*}
	and can be estimated by \eqref{contr}: Combining this with Lemma \ref{unabneu}.\ref{unab1} and \eqref{coef1}, we obtain
	\begin{align*}
	\left\|S(\rho_1) \partial^2_\sigma \rho_j \right. & \left. (\partial_\sigma^3 \rho_1 - \partial_\sigma^3 \rho_2) \right\|_{L_{2, \mu} (J; L_{2} (I))} \leq C(T) \|S(\rho_1)\|_{C(\bar{J}; C((\bar{I}))} \Big(\|\rho_j\|_{L_{\infty}(J, W_{2}^{4(\mu - \nicefrac{1}{2})}(I))} \\
	&~~ + \| \rho_j \|_{L_{2, \mu}(J, W^4_2 (I))} \Big) \left(\|\rho_1 - \rho_2\|_{L_{\infty}(J, W_{2}^{4(\mu - \nicefrac{1}{2})}(I))} \| + \| \rho_1 - \rho_2 \|_{L_{2, \mu}(J, W^4_2 (I))} \right) \\
	&\leq C(\alpha, \Phi^*, \eta, K, T) \left[ C \left(\|\rho_j\|_{ \mathbb{E}_{\mu, T}} + \|\rho_{0}\|_{X_{\mu}} \right) + \|\rho_j\|_{ \mathbb{E}_{\mu, T}} \right] \|\rho_1 - \rho_2 \|_{ \mathbb{E}_{\mu, T}} \\
	&\leq C(\alpha, \Phi^*, \eta, K, \|\rho_{0}\|_{X_{\mu}}, T) \|\rho_1 - \rho_2 \|_{ \mathbb{E}_{\mu, T}},
	\end{align*}
	for $j = 1, 2$ and a $C(\alpha, \Phi^*, \eta, K, T, \|\rho_{0}\|_{X_{\mu}}) \rightarrow 0$ monotonically for $T \rightarrow 0$. The second summand of \eqref{eg} can be treated analogously. Moreover, similar estimates hold true for all the summands except for
	\begin{align*}
	S(\rho_1) \left((\partial^2_\sigma \rho_1)^3 - (\partial^2_\sigma \rho_2)^3 \right) + (S(\rho_1)- S(\rho_2)) \left(\partial^2_\sigma \rho_2\right)^3.
	\end{align*}
	We observe that they can be estimated by using \eqref{contr1} instead of \eqref{contr} together with the estimate on the prefactors \eqref{coef} and \eqref{coef1}, respectively.\\
	
	There is nothing to show for the second component of $\mathcal{F}_T$.
	Finally, we take a look at the last component of $\mathcal{F}_T$: By adding a zero, we have
	\begin{align*}
	G_2(\rho_1) - G_2(\rho_2) 
	&= \underbrace{- \left(b_2(\rho_1) - b_2(\rho_0)\right) \left(\partial_\sigma^3 \rho_1 - \partial_\sigma^3 \rho_2\right)}_{=: I} \underbrace{- \left(b_2(\rho_1) - b_2(\rho_2)\right) \partial_\sigma^3 \rho_2}_{=: II}\\ 
	&~~ \underbrace{- \left(g_2(\rho_1) - g_2(\rho_2)\right)}_{=: III},
	\end{align*}
	where $b_2(\rho) = b_2(\sigma, \rho, \partial_\sigma \rho)$ and $g_2(\rho) = g_2(\sigma, \rho, \partial_\sigma \rho, \partial_\sigma^2 \rho)$, respectively. By
	Lemma \ref{ban-alg}.\ref{lem1}, we deduce for the first summand $I$
	\begin{align*}
	\left\|I \right\|_{W^{\nicefrac{1}{8}}_{2, \mu} (J)} \leq C(T) \left\|b_2(\rho_1) - b_2(\rho_0)\right\|_{W^{\nicefrac{5}{8}}_{2, \mu} (J)} \left\| \partial_\sigma^3 \rho_1 - \partial_\sigma^3 \rho_2 \right\|_{W^{\nicefrac{1}{8}}_{2, \mu} (J)},
	\end{align*}
	where $C(T) \rightarrow 0$ monotonically for $T \rightarrow 0$, since $(b_2(\rho_1) - b_2(\rho_0))_{|t=0} = 0$ by ${\rho_1}_{|t=0} = \rho_0$. Considering the representation of $b_2$, cf.\ \eqref{coeffs}, we can estimate the first factor by combining the results of Claim \ref{lem3a} and Claim \ref{lem4} with the Banach-algebra property of $W^{\nicefrac{5}{8}}_{2, \mu} (J)$, see Lemma \ref{ban-alg}.\ref{lem2}, and Remark \ref{banrem}. Hence, we obtain
	\begin{align*}
	\left\|b_2(\rho_1) \right\|_{W^{\nicefrac{5}{8}}_{2, \mu} (J)} \leq C(\alpha, \Phi^*, \eta, K, \tilde{T}).
	\end{align*}
	Clearly, for $\|b_2(\rho_0)\|_{W^{\nicefrac{5}{8}}_{2, \mu} (J)}$ an analogous estimate holds, as the term does depend on time and, thus, it follows
	\begin{align*}
	\left\|I \right\|_{W^{\nicefrac{1}{8}}_{2, \mu} (J)} 
	\leq C(\alpha, \Phi^*, \eta, K, \tilde{T}, T) \left\| \rho_1 - \rho_2 \right\|_{\mathbb{E}_{\mu, T}}
	\end{align*}
	for $C(\alpha, \Phi^*, \eta, K, \tilde{T}, T) \rightarrow 0$ monotonically for $T \rightarrow 0$. Here, we additionally used the mapping Lemma \ref{unabneu}.\ref{unab5678} to estimate $\| \partial_\sigma^3 \rho_1 - \partial_\sigma^3 \rho_2 \|_{W^{\nicefrac{1}{8}}_{2, \mu} (J)}$.
	
	We proceed with the second summand $II$: In order to estimate the first factor, we use that for $\sigma = 0, 1$ the mapping
	\begin{align*}
	[-\nicefrac{2K_0}{3}, \nicefrac{2K_0}{3}] \times [-\nicefrac{2K_1}{3}, \nicefrac{2K_1}{3}] \ni (\rho, \partial_\sigma \rho) \mapsto b_2(\sigma, \rho, \partial_\sigma \rho)
	\end{align*}
	is Lipschitz-continuous with a Lipschitz constant depending on $\alpha, \Phi^*, \eta$, and $K$. Thus, the estimate
	\begin{align*}
	\| b_2(\rho_1) - b_2(\rho_2)\|_{W^{\nicefrac{5}{8}}_{2, \mu} (J)} 
	&\leq C(\alpha, \Phi^*, \eta, K) \left\| \rho_1 - \rho_2 \right\|_{\mathbb{E}_{\mu, T}} 
	\end{align*}
	is obtained, where we again used the mapping Lemma \ref{unabneu}.\ref{unab5678}. It follows similarly
	to the argumentation for the first summand $I$
	\begin{align*}
	\left\|II \right\|_{W^{\nicefrac{1}{8}}_{2, \mu} (J)} &\leq C(\alpha, \Phi^*, \eta, K, \|\rho_0\|_{X_\mu}, \tilde{T}, T) \left\| \rho_1 - \rho_2 \right\|_{\mathbb{E}_{\mu, T}},
	\end{align*}
	where $C(\alpha, \Phi^*, \eta, K, \|\rho_0\|_{X_\mu}, \tilde{T}, T) \rightarrow 0$ monotonically for $T \rightarrow 0$.
	
	For part $III$, we consider the representation of $g_2$, see \eqref{darg}: By adding a zero
	\begin{align*}
	g_2(\rho_1,) - g_2(\rho_2) 
	&= \underbrace{T(\rho_1) \left( (\partial^2_\sigma \rho_1)^2 - (\partial^2_\sigma \rho_2)^2 \right)}_{=: I_A} + \underbrace{\left( T(\rho_1) - T(\rho_2) \right) (\partial^2_\sigma \rho_2)^2}_{=:II_A} \\ 
	&~~
	+ \underbrace{T(\rho_1) \left( \partial^2_\sigma \rho_1 - \partial^2_\sigma \rho_2 \right)}_{=: I_B} + \underbrace{\left( T(\rho_1) - T(\rho_2) \right) \partial^2_\sigma \rho_2}_{=: II_B} + \underbrace{T(\rho_1) - T(\rho_2)}_{=: II_C},
	\end{align*}
	where $T(\rho) := T(\sigma, \rho, \partial_\sigma \rho)$, see \eqref{t}. The prefactors $T(\rho_i)$, $i = 1, 2$, and $(T(\rho_1) - T(\rho_2))$ can be estimated analogously to $b_2(\rho_i)$ and $(b_2(\rho_1) - b_2(\rho_2))$, respectively. It follows 
	\begin{align}
	\| T(\rho_1) - T(\rho_2)\|_{W^{\nicefrac{5}{8}}_{2, \mu} (J)} &\leq C(\alpha, \Phi^*, \eta, K) \left\| \rho_1 - \rho_2 \right\|_{\mathbb{E}_{\mu, T}}, \label{prev}\\
	\left\|T(\rho_1) \right\|_{W^{\nicefrac{5}{8}}_{2, \mu} (J)} &\leq C(\alpha, \Phi^*, \eta, K, \tilde{T}). \label{prev1}
	\end{align}
	For the summands $II_A, II_B$, and $II_C$ the constant $C(T)$, which tends to zero monotonically for $T \rightarrow 0$, is directly generated by using estimate Lemma \ref{ban-alg}.\ref{lem1}, since $(T(\rho_1) - T(\rho_2))_{|t=0} = 0$ by ${\rho_i}_{|t=0} = \rho_0$ for $i = 1, 2$. Using \eqref{prev}, we deduce
	\begin{align*}
	\left\|\left( T(\rho_1) - T(\rho_2) \right) (\partial^2_\sigma \rho_2)^2 \right\|_{W^{\nicefrac{1}{8}}_{2, \mu} (J)} &\leq C(\alpha, \Phi^*, \eta, K, T) \left\| \rho_1 - \rho_2 \right\|_{\mathbb{E}_{\mu, T}} \left\| (\partial^2_\sigma \rho_2)^2 \right\|_{W^{\nicefrac{1}{8}}_{2, \mu} (J)}, \\
	\left\|\left( T(\rho_1) - T(\rho_2) \right) \partial^2_\sigma \rho_2 \right\|_{W^{\nicefrac{1}{8}}_{2, \mu} (J)} & \leq C(\alpha, \Phi^*, \eta, K, T) \left\| \rho_1 - \rho_2 \right\|_{\mathbb{E}_{\mu, T}} \left\| \partial^2_\sigma \rho_2 \right\|_{W^{\nicefrac{1}{8}}_{2, \mu} (J)},\\
	\left\|T(\rho_1) - T(\rho_2)\right\|_{W^{\nicefrac{1}{8}}_{2, \mu} (J)} &\leq C(T) \left\|T(\rho_1) - T(\rho_2)\right\|_{W^{\nicefrac{5}{8}}_{2, \mu} (J)} \|1\|_{W^{\nicefrac{1}{8}}_{2, \mu} (J)} \left\| \rho_1 - \rho_2 \right\|_{\mathbb{E}_{\mu, T}} \\ 
	&\leq C(\alpha, \Phi^*, \eta, K, \tilde{T}, T) \left\| \rho_1 - \rho_2 \right\|_{\mathbb{E}_{\mu, T}}, 
	\end{align*}
	where the constants $C(\alpha, \Phi^*, \eta, K, \tilde{T}, T) \rightarrow 0$ monotonically for $T \rightarrow 0$. Furthermore, the bounds on $\| (\partial^2_\sigma \rho_2)^2 \|_{W^{\nicefrac{1}{8}}_{2, \mu} (J)}$ and $\| \partial^2_\sigma \rho_2 \|_{W^{\nicefrac{1}{8}}_{2, \mu} (J)}$ follow directly by combining Lemma \ref{unabneu}.\ref{unab5678}, and Lemma \ref{ban-alg}.\ref{lem2a}.
	
	In order to estimate $I_A$ and $I_B$, we exploit the uniform estimate stated in Remark \ref{banrem} together with \eqref{prev1} and deduce
	\begin{align*}
	\left\|T(\rho_1) \left( (\partial^2_\sigma \rho_1)^2 - (\partial^2_\sigma \rho_2)^2 \right) \right\|_{W^{\nicefrac{1}{8}}_{2, \mu} (J)} &\leq C(\alpha, \Phi^*, \eta, K, \tilde{T}) \left\|(\partial^2_\sigma \rho_1 - \partial^2_\sigma \rho_2)(\partial^2_\sigma \rho_1 + \partial^2_\sigma \rho_2) \right\|_{W^{\nicefrac{1}{8}}_{2, \mu} (J)}
	\\
	\left\|T(\rho_1) \left( \partial^2_\sigma \rho_1 - \partial^2_\sigma \rho_2 \right)\right\|_{W^{\nicefrac{1}{8}}_{2, \mu} (J)} &\leq C(\alpha, \Phi^*, \eta, K, \tilde{T}) \left\|\partial^2_\sigma \rho_1 - \partial^2_\sigma \rho_2 \right\|_{W^{\nicefrac{1}{8}}_{2, \mu} (J)} .
	\end{align*}
	Now, we can use the estimate of Lemma \ref{ban-alg}.\ref{lem2a} to generate a constant $C(T) \rightarrow 0$ monotonically for $t \rightarrow 0$. This proves the claim.
\end{proof}

\begin{rem} 
	Note that the formulation of the boundary condition $\partial_{\sigma} \rho (t, \sigma) = 0$ for $\sigma \in \{0, 1\}$ and $t > 0$, \eqref{anglecon}, is very useful for our analysis. It is induced by condition ${\kappa}_{\Lambda} (\sigma) = 0$ for $\sigma \in \{0, 1\}$ on the reference curve. If ${\kappa}_{\Lambda} (\sigma) \neq 0$, we would have to use the boundary condition
	\begin{align*}
	\frac{1}{J(\rho)} \left\langle \Psi_q, \begin{pmatrix} -1 \\ 0 \end{pmatrix} \right\rangle \partial_{\sigma} \rho = - \frac{1}{J(\rho)} \left\langle \Psi_\sigma, \begin{pmatrix} -1 \\ 0 \end{pmatrix} \right\rangle + \cos(\pi - \alpha),
	\end{align*}
	cf.\ \eqref{schlecht}. We observe that the coefficient in front of $\partial_{\sigma} \rho$ itself depends on the first derivative of $\rho$. Since the estimate for the Lipschitz continuity has to be done in the space $W^{\nicefrac{5}{8}}_{2, \mu} (J)$ and both factors can be expected to be elements of this space but do not have higher regularity, we would miss the regularity gap, which we exploited to treat some of the other non-linearities.
\end{rem}

Now, all the tools are available to solve equation \eqref{equ}.
\begin{lem} \label{cml}
	Let the assumptions of Theorem \ref{local} hold true and let $K$ and $\tilde{T}$ be given by Definition \ref{ass}.
	Then there exists a $0 < T < \tilde{T}$, $T = T(\alpha, \Phi^*, \eta, R, \|\mathcal{L}^{-1} \|_{L(\mathbb{E}_{0, \mu} \times \tilde{\mathbb{F}}_{\mu} \times X_{\mu};  \mathbb{E}_{\mu, T})})$ with $\|\rho_0\|_{X_{\mu}} \leq R$, such that there exists a unique solution $\rho \in  \mathbb{E}_{\mu, T}$ for the equation $\mathcal{L}(\rho) = (\mathcal{F}({\rho}), \rho_0)$.
\end{lem}

\begin{proof}
	We consider equation \eqref{equ}:
It is equivalent to the fixed point problem $K ({\rho}) = {\rho}$, where
	\begin{align}
	\mathcal{K}: {\mathcal{B}}_{K, T} &\rightarrow  \mathbb{E}_{\mu, T}, \nonumber \\
	{\rho} &\mapsto \mathcal{K} ({\rho}) := \mathcal{L}^{-1} (\mathcal{F}(\rho), \rho_0).
	\label{fixedpoint}
	\end{align}
	In order to solve the problem using Banach's fixed point theorem, we find an extension of the initial datum $\rho_0$ in the following way: We consider the linearized problem in \eqref{linprob} for the right-hand side $(0, 0, 0, \rho_0)$.
	By Theorem \ref{linex}, there exists a solution ${\widetilde{\rho_0}} \in \mathbb{E}_{\mu, T}$ for $0 < T < T_0$ and we find a constant $C(T_0) \geq \|\mathcal{L}^{-1}\|_{L(\mathbb{E}_{0, \mu} \times \tilde{\mathbb{F}}_{\mu} \times X_{\mu}; \mathbb{E}_{\mu, T})} >0$, such that
\begin{align}
\|{\widetilde{\rho_0}}\|_{ \mathbb{E}_{\mu, T}} \leq C(T_0) \|{\rho_0}\|_{X_{\mu}},
\label{hilfslsg}
\end{align}
for $0 < T < T_0$, cf.\ the Remark \ref{comprem} and Remark \ref{glm}. \\
	
	For showing that $\mathcal{K}: {\mathcal{B}}_{K, T} \rightarrow  \mathbb{E}_{\mu, T}$ is a self-mapping, we consider
	\begin{align*}
\| \mathcal{K} (\rho)\|_{ \mathbb{E}_{\mu, T}} &\leq \left\|\mathcal{L}^{-1} \right\|  \; \left(\|\mathcal{F}({\rho}) - \mathcal{F}(\widetilde{\rho_0})\|_{\mathbb{E}_{0, \mu} \times \tilde{\mathbb{F}}_{\mu}} + \|\mathcal{F}(\widetilde{\rho_0})\|_{\mathbb{E}_{0, \mu} \times \tilde{\mathbb{F}}_{\mu}} + \|\rho_0\|_{X_{\mu}} \right).
\end{align*}
	for every $\rho \in \mathcal{B}_{K,T}$ and $\|\mathcal{L}^{-1} \| := \|\mathcal{L}^{-1} \|_{L(\mathbb{E}_{0, \mu} \times \tilde{\mathbb{F}}_{\mu} \times X_{\mu};  \mathbb{E}_{\mu, T})}$.
	We choose $K > 0$, such that 
	\begin{align}
	\max \left\{C(T_0) \|{\rho_0}\|_{X_{\mu}}, \left\|\mathcal{L}^{-1} \right\| {\|\mathcal{F}(\widetilde{\rho_0})\|_{\mathbb{E}_{0, \mu} \times \tilde{\mathbb{F}}_{\mu}}},
	\left\|\mathcal{L}^{-1} \right\| {\|\rho_0\|_{X_{\mu}}} \right\} &\leq \frac{K}{4}
	\label{hilfhier}
	\end{align}
	for the constant $C(T_0)$ in \eqref{hilfslsg}. 
	Thus, Lemma \ref{contraction} yields
	\begin{align*}
	\| \mathcal{K} (\rho)\|_{ \mathbb{E}_{\mu, T}} &\leq \left\|\mathcal{L}^{-1} \right\| \left( C_{\mathcal{F}_T} \|\rho - \widetilde{\rho_0}\|_{ \mathbb{E}_{\mu, T}} + {\|\mathcal{F}(\widetilde{\rho_0})\|_{\mathbb{E}_{0, \mu} \times \tilde{\mathbb{F}}_{\mu}}} + {\|\rho_0\|_{X_{\mu}}} \right) \\
	&\leq \left\|\mathcal{L}^{-1} \right\| C_{\mathcal{F}_T} K + C_{\mathcal{F}_T} \frac{K}{4} + \frac{K}{4} + \frac{K}{4}, 
	\end{align*}
	where we used ${\rho} \in {\mathcal{B}}_{K,T}$ for $0< T < \max \{T_0, \tilde{T} \}$, cf.\ Definition \ref{ass}, and \eqref{hilfhier}. Moreover, Lemma \ref{contraction} guarantees
	\begin{align}
	\left\|\mathcal{L}^{-1} \right\| C_{\mathcal{F}_T} < \frac{1}{4}\quad \text{and}\quad C_{\mathcal{F}_T}\leq 1
	\label{wichtig}
	\end{align}
	by optionally making $T$ smaller, where we used that $\|\mathcal{L}^{-1}\|$ and $K$ do not depend on $T$, cf.\ Remark \ref{glm}.  
	Consequently, we obtain
	$\mathcal{K} ({\mathcal{B}}_{K, T}) \subset {\mathcal{B}}_{K, T}$.\\
	
	It remains to show that $\mathcal{K}$ is contractive on ${\mathcal{B}}_{K, T}$. 
	Similarly, it holds for all ${\rho}_1, {\rho}_2 \in {\mathcal{B}}_{K, T}$
	\begin{align*}
	\left\| \mathcal{K} ({\rho}_1) - \mathcal{K} ({\rho}_2) \right\|_{ \mathbb{E}_{\mu, T}}
	&\leq \left\| \mathcal{L}^{-1} \right\| \left\|\mathcal{F}({\rho}_1) - \mathcal{F}({\rho}_2)\right\|_{\mathbb{E}_{0, \mu} \times \tilde{\mathbb{F}}_{\mu}}\\
	&\leq \left\|\mathcal{L}^{-1} \right\|C_{\mathcal{F}_T} \left\|{\rho}_1 - {\rho}_2 \right\|_{ \mathbb{E}_{\mu, T}} \leq \frac{1}{4} \left\|{\rho}_1 - {\rho}_2 \right\|_{ \mathbb{E}_{\mu, T}}.
	\end{align*}
	Thus $\mathcal{K}: {\mathcal{B}}_{K, T} \rightarrow {\mathcal{B}}_{K, T}$ is a contraction and by Banach's fixed point theorem there exists a unique fixed point ${\rho}$ of \eqref{fixedpoint} in ${\mathcal{B}}_{K, T}$ for a small enough $T = T(\alpha, \Phi^*, \eta, \|\rho_{0}\|_{X_{\mu}}, \|\mathcal{L}^{-1} \|)$. \\
	
	Assume that $\bar{\rho} \in \mathbb{E}_{\mu, T}$ is another solution to problem \eqref{equ}. Then, we choose $\bar{K} \geq K$, such that $\bar{K} \geq \|\bar{\rho}\|_{\mathbb{E}_{\mu, T}}$. Next, we replace $K$ in the definition of the ball in Definition \ref{ass} by $\bar{K}$ and $\tilde{T}= \tilde{T}(K)$ by
	$\tilde{T}_{*}= \tilde{T}_{*}(\bar{K})$. Then there exists a $T_{*} \in (0, T)$ such that $\mathcal{K}: {\mathcal{B}}_{\bar{K}, T_{*}} \rightarrow {\mathcal{B}}_{\bar{K}, T_{*}}$ is again a contraction. Since ${\rho}$ and $\bar{\rho}$ are fixed points of $\mathcal{K}: {\mathcal{B}}_{\bar{K}, T_{*}} \rightarrow {\mathcal{B}}_{\bar{K}, T_{*}}$, it follows that ${\rho}_{|[0, T_{*}]} = \bar{\rho}_{|[0, T_{*}]}$. Let now
	\begin{align*}
	T_0 = \sup \left\{t \in [T_{*}, T] : {\rho}(t) = \bar{\rho}(t) \textrm{ for all } \tau \leq t \right\}.
	\end{align*}
	If it holds $T_0 < T$, we replace $\rho_0$ by $\bar{\rho}(T_0, \cdot)$ and $t$ by $t - T_0$ in \eqref{equ}. Here, we can use $\bar{\rho}(T_0, \cdot)$ as initial value, since the solution spaces embeds continuously into the temporal trace space, see Lemma \ref{unabneu}.\ref{unab1}, and ${\rho}(T_0, \cdot)$ fulfills the bounds \eqref{kleinrho} and \eqref{kleindrho} for $\nicefrac{2}{3}K_0$ and $\nicefrac{2}{3}K_1$ and we can restart the flow.
	By repeating the previous argument, we obtain ${\rho}_{|[T_0, T_{**}]} = \bar{\rho}_{|[T_0, T_{**}]}$, for $T_{**} \in (T_0, T]$, which contradicts the maximality of $T_0$. Thus, it holds $T_0 = T$ and ${\rho} \equiv \bar{\rho}$.
\end{proof}

The well-posedness-result Theorem \ref{local} follows immediately from Lemma \ref{cml}. We deduce the following Corollary as a direct consequence of Theorem \ref{local}:
\begin{cor}
	\label{orig}
	Let the assumptions of Theorem \ref{local} hold true and let $\rho \in \mathbb{E}_{\mu, T}$ be the unique solution to \eqref{prob} given by Theorem \ref{local}, which fulfills $\rho (\cdot, 0) = \rho_0$ in $X_{\mu}$. Then the function $(t, \sigma) \mapsto \Phi(t, \sigma) = \Psi(\sigma, \rho(t, \sigma))$, see \eqref{kurve}, which is an element of $\mathbb{E}_{\mu, T, \mathbb{R}^2}$, is a solution to \eqref{g1}-\eqref{g3} with $\Phi(0, \cdot) = \Phi(\rho_0)$.
\end{cor}

\begin{rem}
Note that we did not prove a statement about uniqueness of the geometric problem. 
\end{rem}


\footnotesize
\section*{Acknowledgement} 
The results of this paper are part of the second author's PhD Thesis, which was supported by the DFG through the Research Training Group GRK 1692 "Curvature, Cycles, and Cohomology" in Regensburg. The support is gratefully acknowledged.

\end{document}

%% file: Zeichnung1_neu.pdf_tex
\begingroup%
  \makeatletter%
  \providecommand\color[2][]{%
    \errmessage{(Inkscape) Color is used for the text in Inkscape, but the package 'color.sty' is not loaded}%
    \renewcommand\color[2][]{}%
  }%
  \providecommand\transparent[1]{%
    \errmessage{(Inkscape) Transparency is used (non-zero) for the text in Inkscape, but the package 'transparent.sty' is not loaded}%
    \renewcommand\transparent[1]{}%
  }%
  \providecommand\rotatebox[2]{#2}%
  \ifx\svgwidth\undefined%
    \setlength{\unitlength}{235.26792491bp}%
    \ifx\svgscale\undefined%
      \relax%
    \else%
      \setlength{\unitlength}{\unitlength * \real{\svgscale}}%
    \fi%
  \else%
    \setlength{\unitlength}{\svgwidth}%
  \fi%
  \global\let\svgwidth\undefined%
  \global\let\svgscale\undefined%
  \makeatother%
  \begin{picture}(1,0.35575488)%
    \put(0,0){\includegraphics[width=\unitlength]{Zeichnung1_neu.pdf}}%
    \put(0.83965088,0.17951871){\color[rgb]{0,0,0}\makebox(0,0)[b]{\smash{$\alpha$}}}%
    \put(0.39760169,0.33253572){\color[rgb]{0,0,0}\makebox(0,0)[lb]{\smash{$\Gamma_t$}}}%
    \put(0.1765771,0.17951871){\color[rgb]{0,0,0}\makebox(0,0)[lb]{\smash{$\alpha$}}}%
    \put(0.1595752,0.06050551){\color[rgb]{0,0,0}\makebox(0,0)[lb]{\smash{$ \begin{pmatrix} 0 \\ -1\end{pmatrix}$}}}%
    \put(0.09156762,0.23052438){\color[rgb]{0,0,0}\makebox(0,0)[lb]{\smash{${n}_{\Gamma_t}$}}}%
  \end{picture}%
\endgroup%

%% file: param.pdf_tex
\begingroup%
  \makeatletter%
  \providecommand\color[2][]{%
    \errmessage{(Inkscape) Color is used for the text in Inkscape, but the package 'color.sty' is not loaded}%
    \renewcommand\color[2][]{}%
  }%
  \providecommand\transparent[1]{%
    \errmessage{(Inkscape) Transparency is used (non-zero) for the text in Inkscape, but the package 'transparent.sty' is not loaded}%
    \renewcommand\transparent[1]{}%
  }%
  \providecommand\rotatebox[2]{#2}%
  \ifx\svgwidth\undefined%
    \setlength{\unitlength}{293.55bp}%
    \ifx\svgscale\undefined%
      \relax%
    \else%
      \setlength{\unitlength}{\unitlength * \real{\svgscale}}%
    \fi%
  \else%
    \setlength{\unitlength}{\svgwidth}%
  \fi%
  \global\let\svgwidth\undefined%
  \global\let\svgscale\undefined%
  \makeatother%
  \begin{picture}(1,0.19405713)%
    \put(0,0){\includegraphics[width=\unitlength]{param.pdf}}%
    \put(0.57541667,0.13683329){\color[rgb]{0,0,0}\makebox(0,0)[lb]{\smash{$\rho(t, \sigma)$}}}%
    \put(0.36497685,0.16992385){\color[rgb]{0,0,0}\makebox(0,0)[lb]{\smash{$\Gamma_t$}}}%
    \put(0.23568342,0.02423802){\color[rgb]{0,0,0}\makebox(0,0)[lb]{\smash{$\alpha$}}}%
    \put(0.52223256,0.08427848){\color[rgb]{0,0,0}\makebox(0,0)[lb]{\smash{$\Phi^*(\sigma)$}}}%
    \put(0.74640134,0.02154781){\color[rgb]{0,0,0}\makebox(0,0)[lb]{\smash{$\alpha$}}}%
    \put(0.16218341,0.02510184){\color[rgb]{0,0,0}\makebox(0,0)[lb]{\smash{$\alpha$}}}%
    \put(0.8308454,0.02315531){\color[rgb]{0,0,0}\makebox(0,0)[lb]{\smash{$\alpha$}}}%
  \end{picture}%
\endgroup%

%% file: height.pdf_tex
\begingroup%
  \makeatletter%
  \providecommand\color[2][]{%
    \errmessage{(Inkscape) Color is used for the text in Inkscape, but the package 'color.sty' is not loaded}%
    \renewcommand\color[2][]{}%
  }%
  \providecommand\transparent[1]{%
    \errmessage{(Inkscape) Transparency is used (non-zero) for the text in Inkscape, but the package 'transparent.sty' is not loaded}%
    \renewcommand\transparent[1]{}%
  }%
  \providecommand\rotatebox[2]{#2}%
  \ifx\svgwidth\undefined%
    \setlength{\unitlength}{215.55bp}%
    \ifx\svgscale\undefined%
      \relax%
    \else%
      \setlength{\unitlength}{\unitlength * \real{\svgscale}}%
    \fi%
  \else%
    \setlength{\unitlength}{\svgwidth}%
  \fi%
  \global\let\svgwidth\undefined%
  \global\let\svgscale\undefined%
  \makeatother%
  \begin{picture}(1,0.54952665)%
    \put(0,0){\includegraphics[width=\unitlength]{height.pdf}}%
    \put(0.54663831,0.34323015){\color[rgb]{0,0,0}\makebox(0,0)[lb]{\smash{$\Phi^*(\sigma)$}}}%
    \put(0.48816363,0.45689833){\color[rgb]{0,0,0}\makebox(0,0)[lb]{\smash{$\rho(\sigma)$}}}%
  \end{picture}%
\endgroup%